\pdfoutput=1
\documentclass{scrartcl}
\usepackage[intlimits]{amsmath}
\usepackage{amsthm,amssymb,mathtools,bm,ifpdf,relsize,authblk}
\usepackage{url}
\usepackage[abbrev]{amsrefs}
\usepackage{longtable}
\usepackage[T1]{fontenc}
\usepackage{fix-cm}
\usepackage{microtype}
\usepackage{hyperref}

\allowdisplaybreaks[1]

\newtheorem{thm}{Theorem}[section]
\newtheorem*{thm*}{Theorem}
\newtheorem{lemma}[thm]{Lemma}
\newtheorem*{lemma*}{Lemma}
\newtheorem{prop}[thm]{Proposition}
\newtheorem*{prop*}{Proposition}
\newtheorem{cor}[thm]{Corollary}
\newtheorem{thmabc}{Theorem}
\newtheorem{conjabc}[thmabc]{Conjecture}

\theoremstyle{definition}
\newtheorem{ex}[thm]{Example}
\newtheorem*{ex*}{Example}

\newtheorem*{defn*}{Definition}

\newtheorem{rem}[thm]{Remark}
\newtheorem*{rem*}{Remark}

\numberwithin{equation}{section}
\allowdisplaybreaks[1]

\renewcommand{\le}{\leqslant}
\renewcommand{\ge}{\geqslant}

\def\emptyset{\varnothing}

\def\emph{}
\DeclareTextFontCommand{\bfemph}{\bf}
\DeclareTextFontCommand{\itemph}{\it}
\def\emph{\bfemph}

\makeatletter
\def\blankfootnote{\xdef\@thefnmark{}\@footnotetext}

\newcommand*{\textlabel}[2]{%
  \edef\@currentlabel{#1}
  \phantomsection
  #1\label{#2}
}
\makeatother

\newcommand{\reg}{\ensuremath{\mathrm{reg}}}
\newcommand{\UFD}{\textup{\textsf{UFD}}}

\newcommand{\idx}[1]{\lvert#1\rvert}
\newcommand{\divides}[2]{\ensuremath{ {#1} \mid {#2} }}
\newcommand{\onto}{\twoheadrightarrow}
\newcommand{\incl}{\hookrightarrow}

\let\AA\undefined
\newcommand{\AA}{\mathbf{A}}
\newcommand{\acts}{\ensuremath{\curvearrowright}}
\newcommand{\QQ}{\mathbf{Q}}
\newcommand{\NN}{\mathbf{N}}
\newcommand{\ZZ}{\mathbf{Z}}
\newcommand{\CC}{\mathbf{C}}
\newcommand{\RR}{\mathbf{R}}
\newcommand{\fg}{\ensuremath{\mathfrak g}}
\newcommand{\Places}{\ensuremath{\mathcal V}}
\newcommand{\Divisors}{\ensuremath{\mathcal D}}
\newcommand{\noplaces}{\ensuremath{\mathrm g}}
\newcommand{\Zeta}{\ensuremath{\mathsf{Z}}}
\newcommand{\xx}{\ensuremath{\bm x}}
\newcommand{\yy}{\ensuremath{\bm y}}
\newcommand{\fp}{\mathfrak{p}}
\newcommand{\fo}{\mathfrak{o}}
\newcommand{\fO}{\mathfrak{O}}

\newcommand{\cC}{\mathcal{C}}

\newcommand{\Nil}{\ensuremath{\mathsf{N}}}
\newcommand{\Dual}{\ensuremath{\mathsf{A}}}
\newcommand{\Tail}{\ensuremath{\mathsf{B}}}
\newcommand{\CM}{\ensuremath{\mathsf{C}}}

\newcommand{\parts}{\vdash}

\DeclareMathOperator{\rad}{rad}
\DeclareMathOperator{\GL}{GL}
\DeclareMathOperator{\Ker}{Ker}
\DeclareMathOperator{\diag}{diag}
\DeclareMathOperator{\End}{End}
\DeclareMathOperator{\Uni}{U}
\DeclareMathOperator{\Mat}{M}
\DeclareMathOperator{\Spec}{Spec}

\newcommand{\Orth}{\RR_{\ge 0}}
\DeclareMathOperator{\dd}{d\!}

\newcommand{\normal}{\triangleleft}
\newcommand{\dtimes}{\ensuremath{\,\cdotp}}
\newcommand{\card}[1]{\lvert#1\rvert}

\DeclarePairedDelimiter{\abs}{\lvert}{\rvert}

\newcommand{\len}[1]{\ensuremath{\mathrm{len}({#1})}}
\newcommand{\der}{\ensuremath{\mathsf{d}}}

\DeclareMathOperator{\Real}{Re}
\DeclareMathOperator{\Tr}{Tr}

\newcommand{\us}{\ensuremath{\mathsf{\sigma}}}
\newcommand{\ind}[2]{\ensuremath{{#1}^{-1}({#2})}}
\newcommand{\llb}{\ensuremath{[\![ }}
\newcommand{\rrb}{\ensuremath{]\!] }}

\ifpdf
  \hypersetup{pdftitle={Enumerating submodules invariant under an endomorphism},
    pdfauthor={Tobias Rossmann}
  }
\fi
\title{Enumerating submodules invariant under an endomorphism}
\author{Tobias Rossmann}
\affil{\small Fakult\"at f\"ur Mathematik, Universit\"at Bielefeld, D-33501
  Bielefeld, Germany}
\date{}

\begin{document}

\maketitle
\thispagestyle{empty}

\vspace*{-4em}
\begin{abstract}
  \small
  We study zeta functions enumerating submodules
  invariant under a given endomorphism of a finitely generated module over the
  ring of ($S$-)integers of a number field.
  In particular, we compute explicit formulae involving Dedekind zeta functions
  and establish meromorphic continuation of these zeta functions to the  complex plane.
  As an application, we show that ideal zeta functions associated with nilpotent
  Lie algebras of maximal class have abscissa of convergence $2$.
\end{abstract}

\blankfootnote{\noindent{\itshape 2010 Mathematics Subject Classification.}
  11M41, 15A04, 15A21, 17B30


  \medskip
  {\noindent
  This work is supported by the DFG Priority Programme
  ``Algorithmic and Experimental Methods in Algebra, Geometry and Number
  Theory'' (SPP 1489).}}


\section{Introduction}
\label{s:intro}

\paragraph{Zeta functions derived from endomorphisms.}
Throughout, rings are assumed to be commutative and unital.
We say that a ring $R$ has \emph{polynomial submodule
  growth} if the following holds for every finitely generated $R$-module $M$:
for each $m \ge 1$, the number of submodules of additive index $m$ of $M$ is finite and
polynomially bounded as a function of $m$.
Recall that $R$ is \emph{semi-local} if it contains only
finitely many maximal ideals.

\begin{thm}[{\cite[Thm~1]{Seg97}}]
  \label{thm:segal}
  Let $R$ be a ring which is finitely generated over $\ZZ$ or semi-local with
  finite residue fields.
  Then $R$ has polynomial submodule growth if and only if it has Krull dimension at most $2$.
\end{thm}

Let $R$ be a ring with polynomial submodule growth, 
let $M$ be a finitely generated
$R$-module, and let $A \in \End_R(M)$.
For $m \ge 1$, let $a_m(A,R)$ denote the number of $A$-invariant
$R$-submodules $U \le M$ with $\idx{M:U} = m$.
We define a zeta function
\[
\zeta_{A,R}(s) := \sum_{m=1}^\infty a_m(A,R) m^{-s}
\]
and we let $\alpha_{A,R} < \infty$ denote its abscissa of convergence;
it is well-known that $\alpha_{A,R}$ is precisely the degree of polynomial
growth of the partial sums $a_1(A,R) + \dotsb + a_m(A,R)$ as a function of $m$.

The zeta functions $\zeta_{A,R}(s)$ belong to the larger
theory of subobject zeta functions; for a recent survey of the area, see
\cite{VollSurvey}.
Indeed, using the terminology from \cite{topzeta}, $\zeta_{A,R}(s)$ is the
submodule zeta function $\zeta_{R[A] \acts M}(s)$ of the enveloping algebra
$R[A] := \sum\limits_{i=0}^\infty R \dtimes A^i \subset \End_R(M)$ of $A$ acting on $M$.

The main results of this article, Theorems~\ref{Thm:global}--\ref{Thm:poles},
constitute a rather exhaustive analysis of the zeta functions $\zeta_{A,R}(s)$
in the cases that $R$ is the ring of ($S$-)integers of a number field
or a (generic) completion of such a ring.
In particular, our findings provide further evidence in support of the
author's general conjectures on submodule zeta functions stated in \cite[\S 8]{topzeta}.

\paragraph{Related work: invariant subspaces.}
The study of subspaces invariant under an endomorphism has a long history.
For a finite-dimensional vector space $V$ over the real or complex numbers and
$A \in \End(V)$, Shayman~\cite{Sha82}
investigated topological properties of the compact analytic space $S_A$ of $A$-invariant subspaces of
$V$.
In particular, if $A$ is nilpotent, then he found the subspace $S_A(d) \subset
S_A$ of $d$-dimensional $A$-invariant subspaces of $V$ to be connected but
usually singular.

For an arbitrary ground field $F$ and a fixed number $n$, Ringel and
Schmidmeier~\cite{RS08} studied the category of triples $(V,U,T)$, where $V$ is
a finite-dimensional vector space over $F$, $T\in \End_F(V)$ satisfies $T^n =0$,
and $U\le V$ is $F$-invariant.
While their point of view is rather different from ours, we would like to
point out that they found the case of exponent $n \ge 7$ to involve instances of
so-called ``wild'' representation type.

\paragraph{Ideal zeta functions.}
In our study of the zeta functions $\zeta_{A,R}(s)$, we will frequently encounter
another special case of submodule zeta functions, namely ideal zeta functions.
Let $R$ be a ring with polynomial submodule growth and let $\mathsf A$ be a
possibly non-associative $R$-algebra whose underlying $R$-module is finitely
generated.
We write $\mathsf I \normal_R \mathsf A$ to indicate that $\mathsf I$ is a
two-sided ideal of $\mathsf A$ which is also an $R$-submodule.
The \emph{ideal zeta function} (cf.\ \cite{GSS88}) of $\mathsf A$ is
$$\zeta_{\mathsf A}(s) := \sum_{\substack{\mathsf I \normal_R \mathsf A
    \\\idx{\mathsf A:\mathsf I}<\infty}} \idx{\mathsf A:\mathsf I}^{-s}.$$
For example, the ideal zeta function of the ring of integers of a number field $k$
is precisely the Dedekind zeta function of $k$.
In particular, the ideal zeta function of $\ZZ$ is the Riemann zeta function
$\zeta(s)$.
As explained in \cite[Rem.\ 2.2(ii)]{topzeta}, ideal zeta functions are in fact
a special case of the submodule zeta functions discussed below.

\paragraph{Global setup, Euler products, and growth rates.}
For the remainder of this article, let~$k$ be a number field with ring of
integers $\fo$.

Let $\Places_k$ denote the set of non-Archimedean places of $k$.
For $v \in \Places_k$, let $k_v$ be the $v$-adic completion of $k$ and 
let $\fo_v$ be its valuation ring.
For $S\subset \Places_k$,
let $$\fo_S = \bigcap\limits_{v\in\Places_k\setminus S} \fo_v \cap k$$
be the usual ring of $S$-integers of $k$.

In the following, we investigate $\zeta_{A,R}(s)$, where $A\in \End_R(M)$ and
$R = \fo_v$ or $R = \fo_S$ for $v \in \Places_k$ or a finite set $S\subset
\Places_k$, respectively.
The techniques that we use are predominantly local and valid for almost all
places of~$k$ (i.e.\ for all but finitely many places);
the exclusion of a finite number of exceptional places is common and
frequently unavoidable in the theory of subobject zeta functions.

If $M$ is a finitely generate $\fo_S$-module,
then $M \otimes_{\fo_S} \fo_v$ is a free $\fo_v$-module for almost all $v\in
\Places_k\setminus S$.
We thus lose little by henceforth assuming that $M = \fo_S^n$ and
$A \in \Mat_n(\fo_S)$, where $\Mat_n(R)$ denotes the algebra of $n\times n$
matrices over a ring $R$.
Note that if $A \in \Mat_n(k)$, then $A \in \Mat_n(\fo_v)$ for almost all $v \in
\Places_k$.
In order to exclude trivialities, unless otherwise stated, we always assume that $n > 0$.
Being instances of submodule zeta functions,
the zeta functions $\zeta_{A,\fo_S}(s)$ admit natural Euler product factorisations.

\begin{prop*}[Cf.\ {\cite[Lemma~2.3]{topzeta}}]
  Let $A \in \Mat_n(\fo_S)$ for finite $S \subset \Places_k$.
  Then $$\zeta_{A,\fo_S}(s) = \prod\limits_{v \in \Places_k\setminus S} \zeta_{A,\fo_v}(s).$$
\end{prop*}

The following is a consequence of deep results of du~Sautoy and Grunewald on
subobject zeta functions expressible in terms of what they call ``cone integrals''.

\begin{thm}[{Cf.\ \cite[\S 4]{dSG00}}]
  \label{thm:dSG}
Let $A \in \Mat_n(\fo_S)$ for finite $S \subset \Places_k$. Then:
\begin{enumerate}
  \item
  \label{thm:dSG1}
    The abscissa of convergence $\alpha_{A,\fo_S}$ of $\zeta_{A,\fo_S}(s)$ is a rational number.
  \item
  \label{thm:dSG2}
  $\zeta_{A,\fo_S}(s)$ admits meromorphic continuation to $\{ s \in \CC : \mathrm{Re}(s) >
  \alpha_{A,\fo_S} - \delta\}$ for some $\delta > 0$.
  This continued function is regular on the line $\mathrm{Re}(s) = \alpha_{A,\fo_S}$
  except for a pole at $s = \alpha_{A,\fo_S}$.
\item
  \label{thm:dSG3}
  Let $\beta_{A,\fo_S}$ denote the multiplicity of the pole of
  (the meromorphic continuation of) $\zeta_{A,\fo_S}(s)$ at $\alpha_{A,\fo_S}$.
  Then there exists a real constant $c_{A,\fo_S} > 0$ such that
  \[
  a_1(A,\fo_S) + \dotsb + a_m(A,\fo_S) \sim c_{A,\fo_S} \dtimes m^{\alpha_{A,\fo_S}} (\log m)^{\beta_{A,\fo_S}-1}.
  \]
  where $f(m) \sim g(m)$ signifies that $f(m)/g(m) \to 1$ as $m \to \infty$.
\end{enumerate}
\end{thm}

\paragraph{Matrices, polynomials, and partitions.}
Prior to stating our main results, we need to establish some notation and recall
some terminology.
By a \emph{partition} of an integer $n \ge 0$, we mean a non-increasing sequence
$\bm\lambda = (\lambda_1,\dotsc,\lambda_r)$ of positive integers with $n =
\lambda_1 + \dotsb + \lambda_r$;
for background, we refer to \cite{Mac79}.
We write $\abs{\bm\lambda} := n$,
$\len{\bm\lambda} := r$, and $\lambda_{-1} := \lambda_r$.
We write $\bm\lambda \parts n$ to signify that $\bm\lambda$ is a partition of $n$.
For $i \ge 0$, define $\us_i(\bm\lambda) := \lambda_1 + \dotsb + \lambda_i$.
For $1\le j \le \abs{\bm\lambda}$, let $\ind{\bm\lambda} j$ be the unique
number $i \in \{ 1,\dotsc, \len{\bm\lambda} \}$ with $\us_{i-1}(\bm\lambda) < j \le \us_i(\bm\lambda)$;
equivalently,
$\ind{\bm\lambda} j = \min\Bigl( i \in \{ 1,\dotsc, \len{\bm\lambda}\} : j \le \us_i(\bm\lambda)\Bigr)$.
The \emph{dual partition} of $\bm\lambda$ is denoted by $\bm\lambda^*$.
Thus, if $\abs{\bm\lambda} >0$, then $\bm\lambda^* = (\mu_1,\dotsc,\mu_t)$, where $t = \lambda_1$ and
$\mu_i = \#\bigl\{ i \in \{ 1,\dotsc,\len{\bm\lambda} \} : \lambda_i \ge i\bigr\}$.

For a monic polynomial $f = X^m + a_{m-1}X^{m-1} + \dotsb + a_0$,
let
\[
\CM(f) = 
\begin{bmatrix}
0 & 1\\
  & \ddots & \ddots \\
  & & 0 & 1\\
-a_0 & \hdots & -{a_{m-2}} & -a_{m-1}
\end{bmatrix}
\]
be its companion matrix.
Let $A \in \Mat_n(k)$.
It is well-known that there are monic irreducible polynomials $f_1,\dotsc,f_e
\in k[X]$
and partitions $\bm\lambda_1, \dotsc,\bm\lambda_e$ of
positive integers $n_1$,$\dotsc$, $n_e$ such that $n = \deg(f_1) n_1 + \dotsb + \deg(f_e) n_e$ 
and $A$ is similar to its (primary) rational canonical form
\[
\diag\Bigl(
\CM\Bigl(f_1^{\lambda_{1,1}}\Bigr),
\dotsc,
\CM\Bigl(f_1^{\lambda_{1,\len{\bm\lambda_1}}}\Bigr),
\,\,\dotsc\dotsc, \,\,
\CM\Bigl(f_e^{\lambda_{e,1}}\Bigr),
\dotsc\,
\CM\Bigl(f_e^{\lambda_{e_,\len{\bm\lambda_e}}}\Bigr)
\Bigr)
\]
over $k$.
We call $( (f_1,\bm\lambda_1), \dotsc, (f_e,\bm\lambda_e))$ an
\emph{elementary divisor vector} of $A$ over $k$;
any two elementary divisor vectors of $A$ coincide up to reordering.

\paragraph{Main results.}
Recall that $k$ is a number field with ring of integers $\fo$.
Throughout, $\fp_v \in \Spec(\fo)$ denotes the prime ideal corresponding
to a place $v \in \Places_k$ and $q_v = \card{\fo/\fp_v}$ denotes the residue field size
of $k_v$.
Our global main result is the following.

\begin{thmabc}
  \label{Thm:global}
  Let $S\subset \Places_k$ be finite and $A \in \Mat_n(\fo_S)$.
  Let $((f_1,\bm\lambda_1), \dotsc, (f_e,\bm\lambda_e))$ be an elementary divisor
  vector of $A$ over $k$.
  Write $k_i = k[X]/(f_i)$.
  Let $\fo_i$ denote the ring of integers of $k_i$.
  Let $S_i = \{ w \in \Places_{k_i} : \exists v \in S. \divides w v\}$
  and write $\fo_{i,S_i} := (\fo_i)_{S_i}$.
  Then the following hold:
  \begin{enumerate}
    \item
      \label{Thm:global1}
      There are finitely many places $w_1,\dotsc,w_\ell \in \Places_k \setminus S$ and
      associated rational functions $W_1,\dotsc,W_\ell \in \QQ(X)$ such that
      \begin{equation}
        \label{eq:global}
        \zeta_{A,\fo_S}(s) = \prod_{u=1}^\ell W_u(q_{w_u}^{-s}) \times \prod_{i=1}^e
        \prod_{j=1}^{\abs{\bm\lambda_i}} \zeta_{\fo_{i,S_i}}\bigl( (\bm\lambda_i^*)^{-1}(j) \dtimes s - j + 1\bigr).
      \end{equation}
      In particular, $\zeta_{A,\fo_S}(s)$ admits meromorphic continuation to the
      complex plane.
    \item
      \label{Thm:global2}
      The abscissa of convergence $\alpha_{A,\fo_S}$ of $\zeta_{A,\fo_S}(s)$ satisfies
      $\alpha_{A,\fo_S} = \max\limits_{1\le i \le e} \len{\bm\lambda_i} \in \NN$.
    \item
      \label{Thm:global3}
      Let $I := \bigl\{ i \in \{ 1,\dotsc,e\} : \len{\bm\lambda_i} =
      \alpha_{A,\fo_S}\bigr\}$.
      Then the multiplicity $\beta_{A,\fo_S}$ of the pole of $\zeta_{A,\fo_S}(s)$
      at $\alpha_{A,\fo_S}$ satisfies $\beta_{A,\fo_S} = \sum\limits_{i\in I} \lambda_{i,-1}$.
    \end{enumerate}
\end{thmabc}

As we will see, part~(\ref{Thm:global1}) is in fact a consequence of a similar 
formula~\eqref{eq:local} which is valid for almost all local zeta functions $\zeta_{A,\fo_v}(s)$. 
The exceptional factors $W_u(q_{w_u}^{-s})$ in~\eqref{eq:global}
cannot, in general, be omitted, see Example~\ref{ex:exceptional} below.

We note that the special case $A = 0_n$ in Theorem~\ref{Thm:global} is
consistent with the well-known formula $\zeta_{\fo_S}(s) \zeta_{\fo_S}(s-1)
\dotsb \zeta_{\fo_S}(s - (n-1))$ for the zeta function enumerating all
finite-index submodules of $\fo_S^n$.
We further note that the shape of the right-hand side of \eqref{eq:global} is
rather similar to that of Solomon's formula \cite[Thm~1]{Sol77} for the zeta
function enumerating submodules of finite index of a $\ZZ G$-lattice for a
finite group~$G$.

\medskip
Local functional equations under ``inversion of the residue field size'' are a
common, but not universal, phenomenon in the theory of subobject zeta functions;
see \cite{Vol10,Vol16}.
For an extension of number fields $k'/k$ and $v \in \Places_k$, let
$\noplaces_v(k')$\label{lab:noplaces} denote the number of places of $k'$ which divide $v$.

\begin{thmabc}
  \label{Thm:FEqn}
  Let $A \in \Mat_n(k)$ and let $((f_1,\bm\lambda_1),\dotsc,(f_e,\bm\lambda_e))$ be an
  elementary divisor vector of $A$ over $k$.
  Write $\bm\mu_i := \bm\lambda_i^*$.
  Then, for almost all $v \in \Places_k$,
  \begin{equation}
    \label{eq:feqn}
    \zeta_{A,\fo_v}(s) \Bigg\vert_{q_v^{\phantom 1}\to q_v^{-1}}
    = (-1)^{\sum\limits_{i=1}^e \abs{\bm\lambda_i} \dtimes \noplaces_v(k[X]/(f_i))}
    \dtimes
    q_v^{\sum\limits_{i=1}^e \deg(f_i) \binom{\abs{\bm\lambda_i}}{2} -
      \Bigl(\sum\limits_{i=1}^ e \deg(f_i) \sum\limits_{j=1}^{\lambda_{i1}}j
      \mu_{ij}\Bigr)s}
    \dtimes \zeta_{A,\fo_v}(s).
  \end{equation}
\end{thmabc}

The operation of inverting $q_v$ can be interpreted using \eqref{eq:local} or,
in far greater generality, in terms of suitable explicit formulae as
in \cite{Vol10}.
We note that in the special case that $(A - a 1_n)^n = 0$ for some $a \in k$,
the functional equation \eqref{eq:feqn} follows
from \cite[Thm\ 1.2]{Vol16} (see \cite[Rem.\ 1.5]{Vol16}).

\medskip
It is natural to ask what properties of $A$ can be inferred from its associated
zeta functions.
We will make frequent use of the following elementary observation.
\begin{lemma*}
  Let $A,B\in \Mat_n(k)$.
  Suppose that 
  $k[A]$ and $k[B]$ are similar (i.e.\ $\GL_n(k)$-conjugate).
  Then for almost all $v \in \Places_k$,
  $\zeta_{A,\fo_v}(s) = \zeta_{B,\fo_v}(s)$. \qed
\end{lemma*}

The following is another consequence of our explicit formulae.
\begin{thmabc}
  \label{Thm:sim}
  Let $A \in \Mat_n(k)$ and $B \in \Mat_m(k)$ be nilpotent.
 The following are equivalent:
  \begin{enumerate}
  \item
    \label{Thm:sim1}
    $n = m$ and $A$ and $B$ are similar.
  \item
    \label{Thm:sim2}
    For almost all $v \in \Places_k$, $\zeta_{A,\fo_v}(s) =
    \zeta_{B,\fo_v}(s)$.
  \item
    \label{Thm:sim3}
    There exists a finite $S \subset \Places_k$ such that $A$ and $B$ both have
    entries in $\fo_S$ and such that $\zeta_{A,\fo_S}(s) = \zeta_{B,\fo_S}(s)$.
  \end{enumerate}
\end{thmabc}

The nilpotency condition in Theorem~\ref{Thm:sim} cannot, in general, be
omitted, see Remark~\ref{rem:nilpotency_required}.

\medskip
The author previously conjectured \cite[\S 8.3]{topzeta} that generic local submodule zeta functions
associated with nilpotent matrix algebras have a simple pole at zero.
In the present case, our explicit formulae allow us to deduce the following.

\begin{thmabc}
  \label{Thm:poles}
  Let $A \in \Mat_n(k)$.
  Then for almost all $v \in \Places_k$, $\zeta_{A,\fo_v}(s)$ has a pole at zero.
  Moreover, the following are equivalent:
  \begin{enumerate}
  \item For almost all $v \in \Places_k$, $\zeta_{A,\fo_v}(s)$ has a \underline{simple} pole at
    zero.
  \item
    There exists $a \in k$ with $(A - a 1_n)^n = 0$.
  \end{enumerate}
\end{thmabc}

\paragraph{Behaviour at zero in general---a conjecture.}
We use this opportunity to state a generalisation of our conjecture on
the behaviour at zero of local submodule zeta functions (see \cite[Conj.\ IV and
\S 8.3]{topzeta});
this generalisation disposes of the mysterious nilpotency assumption found in its precursor.

For a ring $R$ with polynomial submodule growth, a finitely generated
$R$-module $M$, and $\Omega \subset \End_{R}(M)$,
the submodule zeta function $\zeta_{\Omega \acts M}(s)$ is the Dirichlet
series enumerating $\Omega$-invariant $R$-submodules of finite index of $M$
(cf.\ \cite[Def.\ 2.1(ii)]{topzeta}).

Let $V$ be a finite-dimensional vector space over $k$ and let $\mathcal A\subset
\End_k(V)$ be an associative, unital subalgebra.
Let $\rad(\mathcal A)$ denote the (nil)radical of $\mathcal A$.
By the Wedderburn-Malcev Theorem \cite[Thm~72.19]{CR62},
there exists a subalgebra $\mathcal S
\subset \mathcal A$ such that $\mathcal A = \rad(\mathcal A) \oplus
\mathcal S$ as vector spaces (whence $\mathcal S \approx_k \mathcal
A/\rad(\mathcal A)$ is semisimple);
moreover, $\mathcal S$ is unique up to conjugacy under $(1 + \rad(\mathcal A))
\le \mathcal A^\times$.
Choose $\fo$-forms $\mathsf V \subset V$, $\mathsf A \subset \End_{\fo}(\mathsf
V)$ and $\mathsf S \subset \End_{\fo}(\mathsf V)$ of $V$, $\mathcal A$, and
$\mathcal S$, respectively.
We write $\mathsf X_v := \mathsf X \otimes_{\fo} \fo_v$ in the following.

\begin{conjabc}
  \label{conj:zero}
  For almost all $v \in \Places_k$,
  \[
  \frac{\zeta_{\mathsf A_v  \acts \mathsf V_v}(s)}
  {\zeta_{\mathsf S_v  \acts \mathsf V_v}(s)}
  \Biggm\vert_{s=0} = 1.
  \]
\end{conjabc}

This conjecture reduces to the behaviour predicted in \cite[\S 8.3]{topzeta} in
the ``nilpotent case'' $\mathcal A = \rad(\mathcal A) \oplus k 1_V$.
In order to make Conjecture~\ref{conj:zero} more explicit, we recall Solomon's 
formula for $\zeta_{\mathsf S_v \acts \mathsf V_v}(s)$.
Let $\mathcal S = \mathcal S_1 \oplus \dotsb \oplus \mathcal S_r$ be the Wedderburn
decomposition of the semisimple algebra $\mathcal S$ (so that each $\mathcal
S_i$ is simple).
Let $W_i$ be a simple $\mathcal S_i$-module and
decompose $V = V_1 \oplus \dotsb \oplus V_r$, where $V_i$ is isomorphic to
$W_i^{m_i}$ and $\mathcal S$ acts diagonally on $V$.
Let $k_i$ be the centre of $\mathcal S_i$ and let $\fo_i$ be the ring of
integers of $k_i$.
Finally, let $e_i$ be the Schur index of the central simple $k_i$-algebra
$\mathcal S_i$ and define $n_i$ by $\dim_{k_i}(\mathcal A_i)= n_i^2$.
\begin{thm}[{\cite[\S 4]{Sol77}}]
  \label{thm:solomon}
  For almost all $v \in \Places_k$,
  \begin{equation}
    \label{eq:solomon}
    \zeta_{\mathsf S_v \acts \mathsf V_v}(s)
    = \prod_{i=1}^r \prod_{j=1}^{m_i e_i}
    \prod_{\substack{w \in \Places_{k_i}\\\divides w v}}
    \zeta_{\fo_{i,w}}(n_is - j + 1).
  \end{equation}
\end{thm}

The special case $\mathcal A = k[\alpha]$ ($\alpha \in \End_k(V)$) of
Conjecture~\ref{conj:zero} follows from Theorem~\ref{thm:solomon} and
Theorem~\ref{thm:local} below.

For a more abstract interpretation of Conjecture~\ref{conj:zero}, note that we
may identify $\mathcal S$ acting on $V$ with $\mathcal A/\rad(\mathcal A)$
acting (faithfully) on the semi-simplification of $V$ as an $\mathcal A$-module (i.e.\ the
direct sum of the composition factors of $V$ as an $\mathcal A$-module).

\paragraph{Overview.}
In order to derive Theorems~\ref{Thm:global}--\ref{Thm:poles}, we proceed as
follows.
In \S\ref{s:redprimary}, we reduce the computation of $\zeta_{A,\fo_S}(s)$ to
the case that the minimal polynomial of $A$ over $k$ is a power of an
irreducible polynomial.
In \S\ref{s:rednilpotent}, we then further reduce to the case that $A$ is
nilpotent.
The heart of this article, \S\ref{s:nilpotent}, is then devoted to the explicit
determination of $\zeta_{A,\fo_v}(s)$ for nilpotent $A$ and almost all $v \in \Places_k$;
as a by-product, in Theorem~\ref{thm:ZZX}, we compute the ideal zeta
function of the $2$-dimensional ring $\ZZ\llb X\rrb$.
We then combine our findings and derive
Theorems~\ref{Thm:global}--\ref{Thm:poles} in \S\ref{s:proofs}.
Finally, as an application, in \S\ref{s:app}, we use Theorem~\ref{Thm:global} to
compute the abscissae of convergence of some (largely unknown) submodule and
ideal zeta functions.

\subsection*{Acknowledgment}

I would like to thank Christopher Voll for interesting discussions.

\subsection*{\textit{Notation}}
Throughout, $\NN = \{ 1,2,\dotsc\}$ and $\delta_{ij}$ denotes the Kronecker symbol.
The symbol ``$\subset$'' indicates not necessarily proper inclusion.
We use $\approx_R$ to denote both the similarity of matrices over $R$ and the
existence of an $R$-isomorphism.
Matrices act by right-multiplication on row vectors.
Matrix sizes are indicated by single subscripts for square matrices and double
subscripts in general; in particular, $1_n$ and $0_{m,n}$ denote the $n\times n$
identity and $m\times n$ zero matrix, respectively.

We say that a property depending on $S$ holds for sufficiently large finite
$S \subset \Places_k$, if there exists a finite $S_0 \subset
\Places_k$ such that the property holds for all finite $S \subset \Places_k$
with $S \supset S_0$.
Given $v \in \Places_k$,
we write $\abs{\dtimes}_v$ for the $v$-adic absolute value on $k_v$ with
$\abs{\pi}_v = q_v^{-1}$ for $\pi \in \fp_v\setminus \fp_v^2$.

By a $p$-adic field, we mean a finite extension $K$ of the $p$-adic numbers
$\QQ_p$ for some prime $p$.
We let $\fO_K$ denote the valuation ring of $K$ and write $q_K$ for the
residue field size of $K$.
Furthermore, $\nu_K$ and $\abs{\dtimes}_K$ denote the additive valuation and
absolute value on $K$, respectively, normalised such that any uniformiser $\pi$
satisfies $\nu_K(\pi) = 1$ and $\abs{\pi}_K = q^{-1}_K$.
When the reference to $K$ is clear, we occasionally omit the
subscript ``$K$''.

\section{Reduction to the case of a primary minimal polynomial}
\label{s:redprimary}

By the following, up to enlarging $S$,
we may reduce the computation of $\zeta_{A,\fo_S}(s)$ to the case where the
minimal polynomial of $A$ over $k$ is primary (i.e.\ a power of an irreducible
polynomial).

\begin{prop}
  \label{prop:primary}
  Let $A \in \Mat_n(k)$.
  Let $f = f_1 \dotsb f_e$ be a factorisation of the minimal
  polynomial $f$ of $A$ over $k$ into a product of pairwise coprime monic polynomials
  $f_i \in k[X]$.
  Let $A_i \in \Mat_{n_i}(k)$ denote the matrix of $A$ acting on $\Ker(f_i(A))$
  with respect to an arbitrary $k$-basis.
  Then for almost all $v \in \Places_k$, $$\zeta_{A,\fo_v}(s) = \prod_{i=1}^e \zeta_{A_i,\fo_v}(s).$$
\end{prop}

\begin{proof}
  It is well-known that $k^n = \Ker(f_1(A)) \oplus \dotsb \oplus \Ker(f_e(A))$
  is an $A$-invariant decomposition 
  into subspaces of dimensions $n_1,\dotsc,n_e$, say, 
  and $f_i$ is the minimal polynomial of $A_i$.
  We may thus assume that $A = \diag(A_1,\dotsc,A_e)$.
  By the Chinese remainder theorem, for each $i = 1,\dotsc,e$, there exists $g_i \in
  k[X]$ with $g_i \equiv \delta_{ij} \bmod {f_j}$ for $j = 1,\dotsc,e$.
  Hence, $g_i(A) = \diag( \delta_{i1} 1_{n_1},\dotsc,\delta_{ie}
  1_{n_e}) \in k[A]$.
  Choose a finite set $S\subset \Places_k$ with $A_i \in \Mat_{n_i}(\fo_S)$ and $g_i
  \in \fo_S[X]$ for $i = 1,\dotsc,e$.

  Let $v \in \Places_k \setminus S$.
  Write $V := \fo_v^n$. 
  The block diagonal shape of $A$ yields an $A$-invariant decomposition
  $V = V_1 \oplus \dotsb \oplus V_e$ into free $\fo_v$-modules of ranks $n_1,\dotsc,n_e$.
  Note that $A$ acts as~$A_i$ on each $V_i$
  and that each $g_i(A)$ acts as the natural map $V \onto V_i \incl V$.
  Let $U \le V$ be an $\fo_v$-submodule.
  If $U$ is $A$-invariant, then it decomposes as $U = U_1 \oplus \dotsb \oplus
  U_e$ for $A_i$-invariant submodules $U_i \le V_i$. 
  We conclude that $(U_1,\dotsc,U_e) \mapsto U_1 \oplus \dotsb \oplus U_e$
  defines a bijection from
  \[
  \Bigl\{ (U_1,\dotsc,U_e) : U_i \le_{\fo_v} V_i \text{ and } U_i A_i \le U_i \text{ for } i = 1,\dotsc, e
  \Bigr\}
  \]
  onto the set of $A$-invariant submodules of $V$ whence 
  $\zeta_{A,\fo_v}(s) = \zeta_{A_1,\fo_v}(s) \dotsb \zeta_{A_e,\fo_v}(s)$.
\end{proof}

\section{Reduction to the case of a nilpotent matrix}
\label{s:rednilpotent}

Recall that $\CM(f)$ denotes the companion matrix of a polynomial $f$.
Given a partition $\bm\lambda = (\lambda_1,\dotsc,\lambda_r)$, let
$$\Nil(\bm\lambda) := \diag(\CM(X^{\lambda_1}),\dotsc,\CM(X^{\lambda_r})).$$

Suppose that the minimal polynomial of $A \in \Mat_n(k)$ is a power of an
irreducible polynomial $f$;
we then say that $A$ is \emph{($f$-)primary}.
The elementary divisors of $A$ are $f^{\lambda_1},\dotsc,f^{\lambda_r}$ 
for a unique partition $\bm\lambda = (\lambda_1,\dotsc,\lambda_r)$  of $n / \deg(f)$.
We call $\bm\lambda$ the \emph{type} of~$A$.

For an extension $k'/k$ of number fields and $S\subset \Places_k$, define
\[
\Divisors_{k'/k}(S) = \{
w \in \Places_{k'} : \exists v \in S. \divides w v
\}.
\]
Hence, using the notation from Theorem~\ref{Thm:FEqn}, $\#\Divisors_{k'/k}(S) = \sum\limits_{v \in S} \noplaces_v(k')$.

In this section, we prove the following.

\begin{thm}
  \label{thm:rednil}
  Let $f \in k[X]$ be monic and irreducible.
  Let $A \in \Mat_n(k)$ be an $f$-primary matrix of type $\bm\lambda$.
  Let $k' = k[X]/(f)$, and let $\fo'$ be the ring of integers of $k'$.
  Then for almost all $v \in \Places_k$,
  \[
  \zeta_{A,\fo_v}(s) = \prod_{\substack{w \in \Places_{k'} \\\divides w v}} \zeta_{\Nil(\bm\lambda),\fo'_w}(s).
  \]
  Hence, for all sufficiently large finite $S \subset \Places_k$,
  setting $S' = \Divisors_{k'/k}(S)$.
  \[
  \zeta_{A,\fo_S}(s) = \zeta_{\Nil(\bm\lambda),\fo'_{S'}}(s).
  \]
\end{thm}

\begin{rem}
  In \cite[\S 3]{Sha82}, the study of the variety of subspaces invariant under an
  endomorphism of a finite-dimensional real or complex vector space is
  reduced to the case of a nilpotent endomorphism.
  Shayman proceeds by first reducing to the case of a primary
  endomorphism (\cite[Thm~2]{Sha82}) and our Proposition~\ref{prop:primary}
  proceeded along the same lines.
  In his setting, the minimal polynomial of a primary endomorphism is a power of
  a linear or quadratic irreducible and he considers these cases separately.
  His reasoning is similar to arguments employed in our proof of
  Theorem~\ref{thm:rednil} below.
  We may regard the factorisation of $\zeta_{A,\fo_v}(s)$ obtained by
  combining Proposition~\ref{prop:primary} and Theorem~\ref{thm:rednil} 
  as an arithmetic analogue of the factorisation of the space of $A$-invariant
  subspaces in \cite[Thm~3]{Sha82}.
  In \cite[\S 4]{Sha82}, Shayman then proceeds to study invariant subspaces of
  nilpotent matrices in Jordan normal form.
  For our purposes, a slightly different normal form, introduced in
  \S\ref{ss:nf}, will prove advantageous.
\end{rem}

Our proof of Theorem~\ref{thm:rednil} requires some preparation.

\subsection{A generalised Jordan normal form for primary matrices}

Let $\otimes$ denote the usual Kronecker product $[a_{ij}]\otimes B = [a_{ij}B]$
of matrices.
The following result is a special case of the ``separable Jordan
normal form'' in \cite[\S 6.2]{Nor12};
it can also be obtained by restriction of scalars from the usual Jordan normal
form of an $f$-primary matrix over a minimal splitting field of $f$ over $k$.

\begin{prop}
  \label{prop:jnf}
  Let $f \in k[X]$ be monic and irreducible of degree $d$.
  Let $A \in \Mat_n(k)$ be $f$-primary of type $\bm\lambda$. 
  Write $m := n/d$.
  Then $A\approx_k 1_m \otimes \CM(f) + \Nil(\bm\lambda) \otimes 1_d$.  
\end{prop}

\begin{lemma}
  \label{lem:diagin}
  Let $f \in k[X]$ be monic and irreducible of degree $d$, $\bm\lambda \parts m > 0$,
  and $A = 1_m \otimes \CM(f) + \Nil(\bm\lambda) \otimes 1_d$.
  Then $1_m \otimes \CM(f) = \diag(\CM(f),\dotsc,\CM(f)) \in k[A]$.
\end{lemma}
\begin{proof}
  Write $\gamma := \CM(f)$ and $e := \lambda_1$;
  note that $X^e$ is the minimal polynomial of $\Nil(\bm\lambda)$ over every field.
  We may naturally regard $A$ as an $m\times m$ matrix over the field $k':=
  k[\gamma]$. Moreover, we may identify $k' = k[1_m \otimes \CM(f)]$ as $k$-algebras.
  Thus, $k[A, 1_m \otimes \CM(f)] = k'[\gamma 1_m + \Nil(\bm\lambda)] =
  k'[\Nil(\bm\lambda)]$
  whence the $k$-dimension of $k[A, 1_m \otimes \CM(f)]$ is $\idx{k':k} e = de$.
  As $f^e$ is the minimal polynomial of $A$ over $k$, 
  the number $de$ is also the $k$-dimension of $k[A]$ whence the claim follows.
\end{proof}

Regarding the transition from the number field $k$ to the local ring $\fo_v$,
we note that the enveloping algebras of companion matrices take the expected
forms over \UFD{}s.
\begin{lemma}
  \label{lem:evalCf}
  Let $R$ be a \UFD{} and let $f \in R[X]$ be monic.
  Then evaluation at $\CM(f)$ induces an isomorphism $R[X]/(f) \approx_R R[\CM(f)]$.
\end{lemma}
\begin{proof}
  Let $K$ denote the field of fractions of $R$.
  The kernel of the natural map $R[X] \to R[\CM(f)]$ is $I := R[X] \cap f
  K[X]$ and, clearly, $f R[X] \subset I$.
  Let $h \in I$ so that $h = fg$ for some $g \in K[X]$.
  By \cite[Thm~7.7.2]{Coh03}, there exists $a \in K^\times$ with
  $af,a^{-1}g \in R[X]$.
  As $f$ is monic (hence primitive), $a \in A$ whence
  $g = a(a^{-1}g) \in R[X]$
  and $h \in fR[X]$.
\end{proof}

\subsection{Properties of $S$-integers and their completions}
\label{ss:S}

\begin{lemma}
  \label{lem:o'}
  Let $k'/k$ be an extension of number fields.
  Let $\fo'$ be the ring of integers of~$k'$.
  Let $S \subset \Places_k$ be finite and $S' = \Divisors_{k'/k}(S)$.
  Then $\fo' \otimes_{\fo} \fo_S \approx_{\fo} \fo'_{S'}$.
\end{lemma}
\begin{proof}
  The following argument is taken from \cite{Con}:
  if $h$ is the class number of $k$ and $a \in \fo$ generates
  the principal ideal $\prod_{v\in S} \fp_v^h$,
  then $\fo_S = \fo[1/a]$.
  We conclude that $\fo' \otimes_{\fo} \fo_S = \fo'[1/a] = \fo'_{S'}$.
\end{proof}

\begin{lemma}
  \label{lem:eqorder}
  Let $f \in k[X]$ be monic and irreducible.
  Let $k' = k[X]/(f)$ with ring of integers $\fo'$.
  Then the following holds for all sufficiently large finite $S \subset \Places_k$:
  \begin{enumerate}
  \item
    \label{lem:eqorder1}
    $\fo_S[X]/(f) \approx_{\fo_S} \fo'_{S'}$,  where $S'= \Divisors_{k'/k}(S)$.
  \item
    \label{lem:eqorder2}
    $\fo_v[X]/(f)
    \approx_{\fo_v}
    \prod\limits_{\substack{w\in\Places_{k'}\\\divides w v}} \fo'_w$
    for $v \in \Places_k\setminus S$.
  \end{enumerate}
\end{lemma}
\begin{proof}
  We freely use the exactness of localisation and completion;
  see \cite[Prop.\ 2.5, Thm~7.2]{Eis95}.
  Let $S_0 \subset \Places_k$ be finite with $f \in \fo_{S_0}[X]$.
  If $S \supset S_0$, then $\fo_{S_0}[X]/(f)
  \otimes_{\fo_{S_0}} \fo_S \approx_{\fo_S} \fo_S[X]/(f)$.
  As $\fo_{S_0}[X]/(f)$ and $\fo'$ both become isomorphic to $k'$ after base change to~$k$, for 
  sufficiently large finite $S \supset S_0$,
  $\fo_S[X]/(f) \approx_{\fo_S} \fo'_{S'}$ by Lemma~\ref{lem:o'}.
  This proves the first part.
  For the second part, first note that, using (\ref{lem:eqorder1}) and Lemma~\ref{lem:o'},
  \begin{equation}
    \label{eq:fov}
    \fo_v[X]/(f)
    \approx_{\fo_v}
    \fo_S[X]/(f) \otimes_{\fo_S} \fo_v
    \approx_{\fo_v}
    \fo'_{S'} \otimes_{\fo_S} \fo_v
    \approx_{\fo_v}
    \fo' \otimes_{\fo} \fo_v.
  \end{equation}

  Write $\fo_{(v)} := \fo_v \cap k$ for the $v$-adic valuation ring of $k$.
  It is easy to see that we may naturally identify $\fo' \otimes_{\fo} \fo_{(v)}$ with the
  integral closure of $\fo_{(v)}$ in $k'$.
  The key observation here is that if $a \in k'$ is a root of a monic
  polynomial $f(X) \in \fo_{(v)}[X]$, then there exists $m \in \fo$ with $v(m) =
  0$ and $ma \in \fo'$.
  Indeed, as in the proof of Lemma~\ref{lem:o'}, we find $m \in \fo$ such that
  for all $w \in \Places_k$,
  $w(m) > 0$ if and only if some coefficient $c$ of $f(X)$ satisfies $w(c) < 0$.
  By replacing $m$ by a suitable power, we can ensure that all coefficients of $m f(X)$
  belong to $\fo$ whence $ma$ is integral over $\fo$ and thus belongs to $\fo'$.

  We conclude (see \cite[Ch.~II, \S 8, Exerc.\ 4]{Neu99})
  that the canonical isomorphism $k' \otimes_k k_v \approx_{k_v}
  \prod\limits_{\divides w v} k'_w$ (\cite[Ch.~II, Prop.\ 8.3]{Neu99}) induces an isomorphism
  $\fo' \otimes_{\fo} \fo_v
  \approx_{\fo_v}
  \prod\limits_{\divides w v} \fo'_w$.
  Part (\ref{lem:eqorder2}) thus follows from the latter isomorphism and \eqref{eq:fov}.
\end{proof}

\subsection{Proof of Theorem~\ref{thm:rednil}}

Recall that $a_m(A,R)$ denotes the number of $A$-invariant $R$-submodules of
$R^n$ of index~$m$, where $A \in \Mat_n(R)$.

\begin{prop}
  \label{prop:ringprod}
  Let $R_1,\dotsc,R_r$ be rings with polynomial submodule growth.
  \begin{enumerate}
    \item
      \label{prop:ringprod1}
      $R := R_1 \times \dotsb \times R_r$ has polynomial submodule growth.
    \item
      \label{prop:ringprod2}
      (Cf.\ \cite[Lem.\ 1]{Sol77}.)
      Let $A \in \Mat_n(R)$ and
      let $A_i$ denote the image of $A$ under the map $\Mat_n(R) \to \Mat_n(R_i)$
      induced by the projection $R \to R_i$.
      Then $a_m(A,R) = a_m(A_1,R_1) \dotsb a_m(A_r,R_r)$ for each $m \in \NN$.
      Thus, $\zeta_{A,R}(s) = \zeta_{A_1,R_1}(s) \dotsb \zeta_{A_r,R_r}(s)$.
    \end{enumerate}
\end{prop}
\begin{proof}
  Decompose $R^n = R_1^n \times \dotsb \times R_r^n$ with $R$ acting diagonally
  on $R^n$.
  Multiplication by $e_i = (\delta_{1i},\dotsc,\delta_{ni}) \in R$ acts as the natural
  map $R^n \to R_i^n \to R^n$.
  Given an $R_i$-submodule $U_i \le R_i^n$ for $i = 1,\dotsc,r$, we
  obtain an $R$-submodule $U = U_1 \times \dotsb \times U_r$ of $R^n$ and
  it is easy to see that every $R$-submodule of $R^n$ is of this form in a
  unique way.
  Evidently, $U$ has finite index in $R^n$ if and only if each $U_i$ has finite
  index in $R_i^n$.
  Part (\ref{prop:ringprod1}) is immediate and (\ref{prop:ringprod2}) follows since $A$ acts as $A_i$ on $R_i^n$.
\end{proof}

\begin{proof}[{Proof of Theorem~\ref{thm:rednil}}]
  Assuming that the finite set $S \subset \Places_k$ is sufficiently large, 
  we can make the following assumptions for all $v \in \Places_k \setminus S$:
  \begin{itemize}
  \item[\texttt{(NOR)}] $A = 1_m \otimes \CM(f) + \Nil(\bm\lambda) \otimes 1_d \in \Mat_n(\fo_v)$ for
    $d = \deg(f)$ and
    $\bm\lambda \parts m$ (Proposition~\ref{prop:jnf}).
  \item[\texttt{(DIA)}] $1_m \otimes \CM(f) \in \fo_v[A]$ (Lemma~\ref{lem:diagin}).
  \item[\texttt{(INT)}]
    $\fo_v[X]/(f) \approx_{\fo_v} \prod\limits_{\substack{w \in \Places_{k'}\\ \divides w v}} \fo'_w$
    (Lemma~\ref{lem:eqorder}).
  \end{itemize}
  Let $v \in \Places_k\setminus S$.
  First note that as an $\fo_v$-module, $\fo_v[\CM(f)]$ is freely generated 
  by $(1_d,\CM(f),\dotsc,\CM(f)^{d-1})$.
  It follows easily that $\fo_v^n$ is free of rank $m$ as an $\fo_v[\CM(f)]$-module.

  Using Lemma~\ref{lem:evalCf},\texttt{(INT)} allows us to identify
  $\fo_v[\CM(f)] = \fo_v[X]/(f) = \prod_{\divides w v} \fo'_w =: R_v$.
  Thanks to \texttt{(NOR)}, we may then regard $A$ as an $m \times m$ matrix over $R_v$.
  It follows from  \texttt{(DIA)} that $A$-invariant $\fo_v$-submodules of
  $\fo_v^n$ coincide with $A$-invariant $R_v$-submodules of $R_v^m$.
  Using \texttt{(DIA)} once more, the latter $R_v$-submodules are precisely those
  invariant under $A - \CM(f) \dtimes 1_m = \Nil(\bm\lambda)$.
  Therefore, $\zeta_{A,\fo_v}(s) = \zeta_{\Nil(\bm\lambda),R_v}(s)$.
  Noticing that the $(0,1)$-matrix $\Nil(\bm\lambda)$ is preserved
  by each projection $R_v \to \fo'_w$, Proposition~\ref{prop:ringprod} shows that 
  $\zeta_{\Nil(\bm\lambda),R_v}(s) = \prod\limits_{\divides w v} \zeta_{\Nil(\bm\lambda),\fo'_w}(s)$
  which concludes the proof.
\end{proof}

\section{The case of a nilpotent matrix}
\label{s:nilpotent}

Let $\bm\lambda \parts n$. 
Recall the definitions
of $\ind{\bm\lambda} j$ from the introduction
and of $\Nil(\bm\lambda)$ from \S\ref{s:rednilpotent}.
\begin{defn*}
$W_{\bm\lambda}(X,Y) = 1 / \prod\limits_{j=1}^n \bigl( 1 - X^{j-1}
Y^{\ind{\bm\lambda}j}\bigr)
\in \QQ(X,Y)$.
\end{defn*}
Equivalently, $W_{\bm\lambda}(X,Y) = 1
/\prod\limits_{i=1}^{\len{\bm\lambda}}\prod\limits_{j=1}^{\lambda_i}
\bigl(1-X^{\us_{i-1}(\bm\lambda)+j-1} Y^i\bigr)$.
This section is devoted to proving the following.

\begin{thm}
  \label{thm:nilpotent}
  Let $\bm\lambda \parts n$ and
  let $K$ be a $p$-adic field.
  Then $$\zeta_{\Nil(\bm\lambda^*),\fO_K}(s) = W_{\bm\lambda}(q_K,q_K^{-s}).$$
\end{thm}

Prior to giving a proof of Theorem~\ref{thm:nilpotent}, 
we record a few consequences.

\begin{cor}
  \label{cor:nilpotent_global}
  Let $A \in \Mat_n(k)$ be nilpotent of type $\bm\lambda$ (see \S\ref{s:rednilpotent}).
  Then for all sufficiently large finite sets $S \subset \Places_k$,
  \[
  \zeta_{A,\fo_S}(s) = \prod\limits_{j=1}^n \zeta_{\fo_S}\Bigl(\ind{(\bm\lambda^*)}
  j \dtimes s - j + 1\Bigr).
  \]
  If $A \in \Mat_n(\fo)$ and $A \approx_{\fo} \Nil(\bm\lambda)$,
  then we may take $S = \emptyset$.
  \qed
\end{cor}

As an application, we can determine the ideal zeta function of
$\ZZ[X]/(X^n)$. Recall that $\zeta(s)$ denotes the Riemann zeta function.

\begin{cor}
  \label{cor:ideals_Xn}
  For every prime $p$,
  $$\zeta_{\ZZ_p[X]/(X^n)}(s) = 1 / \prod\limits_{j=1}^n(1-p^{j-1 - js}).$$
  In particular,
  $$\zeta_{\ZZ[X]/(X^n)}(s) = \prod\limits_{j=1}^n
  \zeta(js - j + 1).$$
\end{cor}
\begin{proof}
  The matrix of multiplication by $X$ acting on $\ZZ[X]/(X^n)$ with respect to the
  basis $(1,X,\dotsc,X^{n-1})$, i.e.\ the companion matrix of $X^n$, is
  precisely $\Nil((n))$.
\end{proof}

\begin{rem*}
  The subalgebra zeta functions of $\ZZ_p[X]/(X^n)$ are known only for $n \le 4$
  and sufficiently large primes $p$. Moreover, the author's computation of these
  zeta functions for $n = 4$ relied on fairly involved machine calculations;
  see \cite[\S 9.2]{padzeta}.
  (The formula for $\zeta_{\ZZ_p[X]/(X^4)}(s)$ in \cite{padzeta} takes up
  about a page in total.)
\end{rem*}

Subobject zeta functions over rings other than $\fo_S$ or $\fo_v$ have received
little attention so far. 
We obtain the following.

\begin{thm}
  \label{thm:ZZX}
  \quad
  \begin{enumerate}
  \item
    \label{thm:ZZX1}
    $\ZZ\llb X\rrb$ has polynomial submodule growth.
  \item
    \label{thm:ZZX2}
    $\zeta_{\ZZ\llb X \rrb}(s) = \prod\limits_{j=1}^\infty \zeta(js - j + 1)$
    for $\Real(s) > 1$.
  \end{enumerate}
\end{thm}
\begin{proof}
  It is well-known that the maximal ideals of $\ZZ\llb X \rrb$ are precisely of
  the form $(X,p)$ for a rational prime $p$.
  It follows that $X$ acts nilpotently on every $\ZZ\llb X\rrb$-module of finite
  length.
  Hence, if $U \le_{\ZZ\llb X\rrb} \ZZ\llb X\rrb^d$ has finite index,
  then $U$ contains $X^n \ZZ\llb X\rrb^d$ for some $n \ge 1$.
  As $\ZZ\llb X\rrb$ is Noetherian, $U$ thus corresponds to a $\ZZ[X]$-submodule
  of $\ZZ[X]^d/X^n\ZZ[X]^d$.
  In particular, (\ref{thm:ZZX1}) follows since $\ZZ[X]$ has polynomial
  submodule growth by Theorem~\ref{thm:segal}.
  Moreover, Corollary~\ref{cor:ideals_Xn} implies the identity in
  (\ref{thm:ZZX2}) on the level of formal Dirichlet series.

  In order to establish (absolute) convergence, let $s > 1$ be real.
  By well-known facts on infinite products,
  $\prod_{j=1}^\infty \zeta(js-j+1)$ converges (absolutely) if and only if the
  same is true of $F(s) := \sum_{j=1}^\infty (\zeta(js-j+1)-1)$.
  Using the non-negativity of the coefficients of each Dirichlet series
  $\zeta(js-j+1)$, we obtain
  \begin{align*}
    F(s) & = \sum_{j=1}^\infty \sum_{n=2}^\infty n^{j-1}(n^j)^{-s}
    = \sum_{n=2}^\infty g_n n^{-s},
  \end{align*}
  where
  $$g_n := n \dtimes \sum_{\substack{m \ge 2,j \ge 1\\n = m^j}} \frac 1 m.$$
  We see that for $N \ge 2$,
  \begin{align*}
  \sum_{n=2}^N g_n & \le N  \sum_{\substack{m \ge 2, j \ge 1\\m^j \le
      N}} \frac 1 m \le N \sum_{m=2}^N \frac{2 \log N} m = \mathcal O(N (\log
  N)^2)
  = \mathcal O(N^{1+\varepsilon})
  \end{align*}
  for every $\varepsilon > 0$.
  In particular, $F(s)$ and $\zeta_{\ZZ\llb X\rrb}(s)$ both converge for $\Real(s) > 1$.
\end{proof}

\begin{rem}
Note, in particular, that $\zeta_{\ZZ\llb X\rrb}(s)$ has an
essential singularity at $s = 1$ and therefore does not admit meromorphic
continuation beyond its abscissa of convergence.
This illustrates that Theorem~\ref{thm:dSG}(\ref{thm:dSG2}) does not carry over
to general ground rings with polynomial submodule growth.
\end{rem}

In order to prove Theorem~\ref{thm:nilpotent}, we employ the $p$-adic
integration machinery from \cite{GSS88}.
For a ring $R$, let $\Tr_n(R)$ denote the $R$-algebra of upper triangular
$n\times n$-matrices over $R$.
Recall that an element of a ring is \emph{regular} if it is not a zero divisor.
Write $\Tr_n^\reg(R) = \{ \bm x \in \Tr_n(R) : \det(\bm x) \in R \text{ is regular}\}$.
For a $p$-adic field $K$, let $\mu_K$ denote the Haar measure on $K^n$ with
$\mu_K(\fO_K^n) = 1$.

\begin{prop}[{\cite[\S 3]{GSS88}}]
  \label{prop:coneint}
  Let $K$ be a $p$-adic field
  and $A \in \Mat_n(\fO_K)$.
  Define
  $V_K(A) := \bigl\{
  \bm x \in \Tr^\reg_n(\fO_K) : \fO_K^n \bm x A \subset \fO_K^n \bm x
  \bigr\}$
  to be the set of upper-triangular $n\times n$ matrices over~$\fO_K$ whose rows
  span an $A$-invariant $\fO_K$-submodule of finite index of $\fO_K^n$.
  Then
  \begin{equation}
    \label{eq:GSS}
    \zeta_{A,\fO_K}(s) = (1-q_K^{-1})^{-n} \int_{V_K(A)} \abs{x_{11}}_K^{s-1}
    \abs{x_{22}}_K^{s-2} \dotsb \abs{x_{nn}}_K^{s-n} \dd\mu_K(\bm x).
  \end{equation}
\end{prop}

\paragraph{Strategy.}
In order to prove Theorem~\ref{thm:nilpotent}, we proceed as follows.
First, in \S\ref{ss:nf}, we define a matrix $\Dual(\bm\lambda)$ which is similar
(over $\ZZ$) to $\Nil(\bm\lambda^*)$ so that $\zeta_{\Nil(\bm\lambda^*),\fO_K}(s) =
\zeta_{\Dual(\bm\lambda),\fO_K}(s)$.
As we will see in \S\ref{ss:recursion},
the advantage of $\Dual(\bm\lambda)$ over $\Nil(\bm\lambda^*)$ is that the sets
$V_K(\Dual(\bm\lambda))$ in Proposition~\ref{prop:coneint} exhibit a natural,
recursive structure.
Specifically, we will define $\der{\bm\lambda} := (\lambda_2,\dotsc,\lambda_{\len{\bm\lambda}})$
and find that $V_K(\Dual(\bm\lambda))$ can be described in terms of
$V_K(\Dual(\der{\bm\lambda}))$ and membership conditions for generic vectors
in generic sublattices. 
In \S\ref{ss:membership},
the geometry of such membership conditions is elucidated by means of suitable
(birational) changes of coordinates.
Finally, in \S\ref{ss:proof_nilpotent}, we combine all these ingredients and
prove Theorem~\ref{thm:nilpotent}.

\subsection{A dual normal form for nilpotent matrices}
\label{ss:nf}

\begin{defn*}
  Let $\bm\lambda = (\lambda_1,\dotsc,\lambda_r) \parts n \ge 0$.
  Define $\der{\bm\lambda} := (\lambda_2,\dotsc,\lambda_r)$.
  We recursively define $\Dual(\bm\lambda) \in \Mat_n(\ZZ)$ as follows:
  \begin{enumerate}
  \item If $r \le 1$, define $\Dual(\bm\lambda) = 0_n$.
  \item
    If $r > 1$,
    define
    \sbox0{$\begin{matrix}1&2\\4&5\end{matrix}$}
    \sbox1{$\begin{matrix}1\\2\end{matrix}$}
    \sbox2{$\begin{matrix} 1_{\lambda_2}  \\ 0_{\lambda_1-\lambda_2,\lambda_2} \end{matrix}$}
    \sbox3{$\begin{matrix} 1 & 2 & 3\\
        1 & 2 & 3 \\
        1 & 2 & 3 \end{matrix}$}
    \begin{equation}
      \label{eq:A}
    \Dual(\bm\lambda) =
    \left[
    \begin{array}{c|cc}
      \makebox[\wd0]{\large $0_{\lambda_1}$}
      & \usebox{2}  & 0_{\lambda_1,\lambda_3+\dotsb+\lambda_r}    \\\hline
      & 
      \vphantom{\usebox{3}}
      \makebox{\large $\Dual(\der{\bm\lambda})$}
    \end{array}
    \right].
    \end{equation}
  \end{enumerate}
\end{defn*}

In other words,

\begin{equation}
  \label{eq:A_alt}
  \Dual(\bm\lambda) = 
  \left[
\begin{array}{ccccc}
  \sbox0{$\begin{matrix}1&2\\4&5\\7&8\end{matrix}$}
  \vphantom{\usebox{0}} \makebox{$0_{\lambda_1}$} &
  \begin{array}{|c|} \hline \makebox{$1_{\lambda_2}$}\\ \hline \makebox{$0_{\lambda_1-\lambda_2,\lambda_2}$}\\\hline\end{array}
  \\
  & \vphantom{\usebox{0}} \makebox{$0_{\lambda_2}$} &
  \begin{array}{|c|} \hline \makebox{$1_{\lambda_3}$}\\ \hline \makebox{$0_{\lambda_2-\lambda_3,\lambda_3}$}\\\hline\end{array}
  \\
  & & \ddots & \ddots \\
  & &&\ddots &
  \begin{array}{|c|} \hline \makebox{$1_{\lambda_r}$}\\ \hline \makebox{$0_{\lambda_{r-1}-\lambda_r,\lambda_r}$}\\\hline\end{array}
  \\
  & & & & \vphantom{\usebox{0}}\makebox{$0_{\lambda_r}$}
\end{array}
\right]
\end{equation}

By the following, the $\Dual(\bm\lambda)$ parameterise similarity
classes of nilpotent matrices.

\begin{prop}
  \label{prop:permconj}
  $\Dual(\bm\lambda^*)$ and $\Nil(\bm\lambda)$ are conjugate by permutation matrices.
\end{prop}
\begin{proof}
  Let $T(\bm\lambda)$ be the Young diagram of $\bm\lambda$ and
  let $V(\bm\lambda)$ be the $\ZZ$-module freely generated by the cells
  of $T$;
  we use ``English notation'' for $T(\bm\lambda)$ and draw each row 
  underneath its predecessor (if any).
  Define $\Theta(\bm\lambda)$ to be the endomorphism of $V(\bm\lambda)$ (acting on the
  right) which sends each cell to its right neighbour if it exists and to zero otherwise.

  We consider two orderings on the cells of $T(\bm\lambda)$ and describe the
  associated matrices representing $\Theta(\bm\lambda)$.
  The \itemph{horizontal order} is defined by traversing the cells of $T(\bm\lambda)$
  from left to right within each row, proceeding from top to bottom.
  Clearly, $\Nil(\bm\lambda)$ is the matrix of $\Theta(\bm\lambda)$ with respect to
  this order.

  The \itemph{vertical order} is obtained by traversing the cells of $T(\bm\lambda)$ 
  from top to bottom within each column, proceeding from left to right.
  Write $\bm\mu := \bm\lambda^*$, say $\bm\mu = (\mu_1,\dotsc,\mu_\ell)$.
  We now show by induction on $\ell$ that the matrix of $\Theta(\bm\lambda)$ with respect 
  to the vertical order is $\Dual(\bm\mu)$---it then follows, in particular,
  that $\Dual(\bm\mu)$ and $\Nil(\bm\lambda)$ are conjugate as claimed. 
  
  If $\ell \le 1$, then $\Theta(\bm\lambda) = 0$ and $\Dual(\bm\mu) = 0$ so let $\ell > 1$.
  Let $t_1,\dotsc,t_n$ be the cells of $T(\bm\lambda)$ according to the vertical order.
  Then $t_i \Theta(\bm\lambda) = t_{\mu_1+i}$ for $1\le i \le \mu_2$
  and $t_i \Theta(\bm\lambda) = 0$ for $\mu_2 < i \le \mu_1$.
  Let $\tilde{\bm\lambda} := (\der{\bm\mu})^*$
  and $\tilde V :=\ZZ t_{\mu_1+1} \oplus \dotsb \oplus \ZZ t_n$.
  We may naturally identify the endomorphism of $\tilde V$ induced by
  $\Theta(\bm\lambda)$ with $\Theta(\tilde{\bm\lambda})$ acting on $V(\tilde{\bm\lambda})$;
  the defining basis of $\tilde V$ is then ordered vertically.
  By induction, the matrix of $\Theta(\bm\lambda)$ acting on $\tilde V$
  with respect to the basis $(t_{\mu_1+1},\dotsc,t_n)$ is therefore
  $\Dual(\der{\bm\mu})$ whence the claim follows from the recursive description of
  $\Dual(\bm\mu)$ in \eqref{eq:A}.
\end{proof}

For $\abs{\bm\lambda} > 0$, let $\Tail(\bm\lambda) \in
\Mat_{\abs{\bm\lambda},\abs{\der{\bm\lambda}}}(\ZZ)$ denote the matrix obtained 
by deleting the first $\lambda_1$ columns of $\Dual(\bm\lambda)$.
The following consequence of \eqref{eq:A_alt} will be useful below.

\begin{lemma}
  \label{lem:B}
  $\Tail(\bm\lambda)$ contains precisely $\lambda_1$ zero rows and
  by deleting these, the $\abs{\der{\bm\lambda}}\times\abs{\der{\bm\lambda}}$ identity matrix is
  obtained. \qed
\end{lemma}

\subsection{Recursion}
\label{ss:recursion}

In this subsection, we give a recursive description of $V_K(\Dual(\bm\lambda))$
(see Proposition~\ref{prop:coneint}).

\begin{lemma}
  \label{lem:rec}
  Let $\bm\lambda = (\lambda_1,\dotsc,\lambda_r) \parts n$
  and let $X$ be the generic upper triangular $n\times n$ matrix.
  Partition $X$ in the form
  \[
  X = \left[\begin{array}{c|c}
      \makebox{$\begin{matrix} X^{\mathrm{I}}_{\lambda_2} &
          * \\ 0_{\lambda_1-\lambda_2,\lambda_2} &
          X^{\mathrm{II}}_{\lambda_1-\lambda_2}\end{matrix}$}
      &  \makebox{$\bar X_{\lambda_1,\abs{\der{\bm\lambda}}}$}\\\hline
      0 & \makebox{$X^{\prime}_{\abs{\der{\bm\lambda}}}$}
    \end{array}\right],
  \]
  where subscripts are added to denote block sizes.
  Then
  \[
  X \Dual(\bm\lambda) =
  \left[
    \begin{array}{c|c}
      \makebox{$0_{\lambda_1}$} & 
      \begin{array}{c|c}
        \begin{array}{c}
          \makebox{$X^{\mathrm I}$} \\
          \makebox{$0$}
        \end{array} &
        \makebox{$\bar X \Tail(\der{\bm\lambda})$}
      \end{array}
      \\\hline
      0 & X' \Dual(\der{\bm\lambda})
    \end{array}
  \right].
  \]
\end{lemma}
\begin{proof}
  This follows easily from \eqref{eq:A}.
\end{proof}

By Lemmas \ref{lem:B}--\ref{lem:rec},
the $\lambda_1\times \abs{\der{\bm\lambda}}$ submatrix obtained by considering the first $\lambda_1$ rows of
$X\Dual(\bm\lambda)$ and then deleting the first $\lambda_1$ columns is of the
form
\begin{equation}
  \label{eq:Xlambda}
X^{\bm\lambda} := 
\begin{bmatrix}
  x_{1,1} & \hdots & x_{1,\lambda_2} & * & \hdots & *\\
  & \ddots  & \vdots & \vdots & \ddots & \vdots\\
  & & x_{\lambda_2,\lambda_2} & * & \hdots & * \\
  & & & \vdots & \ddots & \vdots \\
  & & & * & \hdots & *
\end{bmatrix},
\end{equation}
where the entries marked ``$*$'' indicate unspecified but \itemph{distinct} variables taken from $\bar X$.

\begin{cor}
  \label{cor:rec}
  Let $\bm\lambda \parts n$ and let $K$ be a $p$-adic field.
  For $\bm x \in \Tr_n(K)$, define $\bm x'$ and $\bm x^{\bm\lambda}$ by specialising
  $X'$ and $X^{\bm\lambda}$ from Lemma~\ref{lem:rec} and \eqref{eq:Xlambda},
  respectively, at $\bm x$.
  Then
  \begin{align}
    \label{eq:Vrec}
    V_K(\Dual(\bm\lambda)) = \Bigl\{ \bm x \in \Tr^\reg_n(\fO_K) : 
    \text{(i) } & \text{each row of $\bm x^{\bm\lambda}$ belongs to
      $\fO_K^{\abs{\der{\bm\lambda}}} \bm x '$
      and } \nonumber \\ \text{ (ii) }&
    \bm x'  \in V_K(\Dual(\der{\bm\lambda}))
    \Bigr\}.
  \end{align}
\end{cor}
\begin{proof}
  Let $\bm x \in \Tr_n^\reg(\fO_K)$.
  Clearly,
  $\bm x \in V_K(\Dual(\bm\lambda))$ if and only if every row of $\bm x
  \Dual(\bm\lambda)$ is contained in the $\fO_K$-span of the rows of $\bm x$.
  By Lemma~\ref{lem:rec} and since $\det(\bm x) \not= 0$, the first $\lambda_1$
  rows of $\bm x \Dual(\bm\lambda)$ satisfy this condition if and only if every row
  of $\bm x^{\bm\lambda}$ is contained in the $\fO_K$-span of the rows of $\bm x'$.
  Similarly, the rows numbered $\lambda_1 + 1,\dotsc,n$ of
  $\bm x \Dual(\bm\lambda)$ are contained in the $\fO_K$-span of $\bm x$ if and
  only if each row of $\bm x' \Dual(\der{\bm\lambda})$ is contained in the
  $\fO_K$-span of $\bm x'$ or, equivalently, if $\bm x' \in V_K(\Dual(\der{\bm\lambda}))$.
\end{proof}

\subsection{Characterising submodule membership}
\label{ss:membership}

Condition (i) in \eqref{eq:Vrec} leads us to investigate pairs $(\xx,\yy) \in R^n
\times \Tr_n(R)$ (where $R$ is a ring) such that $\bm x$ is contained in the row
span of $\yy$ over $R$.
In this subsection, we study the set of all such pairs $(\xx,\yy)$ 
in the case that $R = \fO_K$ for a $p$-adic field $K$.

We write $\AA^n = \Spec(\ZZ[X_1,\dotsc,X_n])$ and $\Tr_n = \Spec(\ZZ[Y_{ij} :
1\le i \le j \le n])$.
Let
\begin{equation}
  \label{eq:EnR}
  E_n(R) := \bigl\{
  (\xx, \yy) \in R^n \times \Tr_n(R) : \xx \in R^n \yy \bigr\}.
\end{equation}
We identify $\AA^n\times\Tr_n = \Spec(\ZZ[X_1,\dotsc,X_n,
Y_{11},\dotsc,$ $Y_{1n},Y_{22}, \dotsc, Y_{nn}])$.
Define
\begin{equation}
  \label{eq:Cn}
  \cC_n := 
  \bigl\{ (\alpha,\omega) \in \Orth^n \times \Tr_n(\Orth) : 
  \omega_{ii} \le \alpha_i \text{ for } 1 \le i \le n 
  \bigr\}.
\end{equation}

For a $p$-adic field $K$,
we extend $\nu_K$ to families of elements of $K$ via $\nu_K(a_1,\dotsc,a_m) = (\nu_K(a_1),\dotsc,\nu_K(a_m))$
and write
\[
\cC_n(K) := \Bigl\{ (\xx,\yy) \in K^n \times \Tr_n(K) : (\nu_K(\xx),\nu_K(\yy))
\in \cC_n\Bigr\}
\subset \fO_K^n \times \Tr_n^\reg(\fO_K).
\]
The following lemma
will play a key role in our proof of Theorem~\ref{thm:nilpotent}.
It shows that away from sets of measure zero, a suitable $\ZZ$-defined change of
coordinates (defined independently of $K$) transforms $E_n(\fO_K)$ into $\cC_n(K)$.

\begin{lemma}
  \label{lem:coc}
  There exist
  \begin{itemize}
  \item closed subschemes $V_n,V_n' \subset \AA^n \times \Tr_n$ of the
    form $f_n = 0$ and $f_n'= 0 $, respectively, where $f_n,f_n' \in \ZZ[\bm X,
    \bm Y]$ are non-zero non-units, and
  \item an isomorphism
    $\varphi_n \colon (\AA^n\times\Tr_n)\setminus V_n \to
    (\AA^n\times\Tr_n)\setminus V_n'$
  \end{itemize}
  such that the following conditions are satisfied:
  \begin{enumerate}
  \item
    \label{lem:coc1}
    For each $p$-adic field $K$, 
    $\varphi_n^K( E_n(\fO_K) \setminus V_n(\fO_K)) = \cC_n(K) \setminus
    V_n'(\fO_K)$,
    where $\varphi_n^K$ denotes the map induced by $\varphi_n$ on $K$-points.
  \item
    \label{lem:coc2}
    The Jacobian determinant of $\varphi_n$ is identically $1$.
  \item
    \label{lem:coc3}
    $\varphi_n$ commutes with (the restriction to its domain of) the projection of $\AA^n \times \Tr_n$ onto
    $\Tr_n$ and (the restriction of) the projection onto the first coordinate of $\AA^n$.
  \end{enumerate}
\end{lemma}
\begin{ex*}[$n=2$]
  Let $K$ be a $p$-adic field; we drop the subscripts ``$K$'' in the following.
  Let $x,y,a,b,c \in \fO$ and suppose that
  $x (ay - bx) a b c\not= 0$.
  Define $y' := y - \frac x a b \in K$ and note that $y' \not= 0$.
  Then $(x,y) \in \fO^2 \dtimes \bigl[\begin{smallmatrix} a & b \\ 0 &
    c\end{smallmatrix}\bigr]$ if and only if $\nu(a) \le \nu(x)$ and
  $(x,y) - \frac x a (a,b) = (0,y') \in \fO (0,c)$; the latter
  condition is equivalent to $\nu(c) \le \nu(y')$ and implies that $y' \in \fO$.
  We see that the map $((x,y),\bigl[\begin{smallmatrix} a & b\\0 & c\end{smallmatrix}\bigr])
  \mapsto ((x,y'),\bigl[\begin{smallmatrix} a & b\\0 &
    c\end{smallmatrix}\bigr])$
  has the properties of $\varphi_2$ stated in Lemma~\ref{lem:coc}.
\end{ex*}
\begin{proof}[Proof of Lemma~\ref{lem:coc}]
  We proceed by induction.
  For $n = 1$, we let $f_1 = f_1' = X_1 Y_{11}$ and define $\varphi_1$ to be the identity.
  Clearly, (\ref{lem:coc1})--(\ref{lem:coc3}) are satisfied.

  Let $n > 1$ and suppose that $\varphi_{n-1}$ with the stated properties has
  been defined.
  Let~$K$ be a $p$-adic field
  and let $(\xx,\yy) \in K^n\times \Tr_n(K)$ with $x_1 y_{11} \not= 0$.
  We again drop the subscripts ``$K$''.
  Gaussian elimination shows that $(\xx,\yy) \in E_n(\fO)$ if and only if the
  following conditions are satisfied:
  \begin{enumerate}
  \item[(a)]
    $x_i, y_{ij} \in \fO$  for $1\le i \le j \le n$,
  \item[(b)]
    $\frac{x_1}{y_{11}} \in \fO$, and
  \item[(c)]
    $\bigl(x_2 - \frac{x_1}{y_{11}} y_{12}, \dotsc, x_n - \frac{x_1}{y_{11}}
    y_{1n}\bigr) \in \fO^{n-1} \dtimes \Bigl[ y_{ij} \Bigr]_{2\le i \le j \le
      n}$.
  \end{enumerate}

  We will now simplify (c) using a change of coordinates.
  For $2 \le j \le n$, let $x_j' := x_j - \frac{x_1}{y_{11}} y_{1j}$.
  Write $x_1' := x_1$ and $\xx' := (x_1',\dotsc,x_n')$.
  Note that $(\xx,\yy) \mapsto (\xx',\yy)$ is an
  automorphism of the complement of $Y_{11} = 0$ in $\AA^n \times \Tr_n$
  and that the Jacobian determinant of this map is identically $1$.

  Assuming that $y_{ij} \in \fO$ for $1\le i \le j \le n$ 
  and $\frac{x_1}{y_{11}} \in \fO$, we see that $x_j \in \fO$ if and only if
  $x_j' \in \fO$.
  Hence, $(\xx,\yy) \in E_n(\fO)$ if and only if (b) and the following two
  conditions are satisfied:
  \begin{enumerate}
    \item[(a')] $x_i', y_{ij} \in \fO$  for $1\le i \le j \le n$,
    \item[(c')] $(x_2', \dotsc, x_n') \in \fO^{n-1}  \dtimes \Bigl[ y_{ij}
      \Bigr]_{2\le i \le j \le n}$.
  \end{enumerate}
  After excluding suitable hypersurfaces,
  our inductive hypothesis allows us to perform another change of coordinates,
  replacing $x_2', \dotsc,x_n'$ by $x_2'', \dotsc,x_n''$, say, such that
  $(\xx,\yy) \in E_n(K)$ if and only if the following conditions are
  satisfied:
  \begin{enumerate}
    \item[(a'')] $x_i'', y_{ij} \in \fO$  for $1\le i \le j \le n$ (where
      $x_1'' := x_1' = x_1$) and
    \item[(c'')] $\nu(y_{ii}) \le \nu(x_i'')$ for $1 \le i \le n$;
  \end{enumerate}
  note that (b) is implied by the case $i = 1$ of (c'').

  For (\ref{lem:coc1}), assuming that the product of all $x_i''$ and $y_{ij}$ is
  non-zero, conditions (a'') and (c'') are both satisfied if and only
  if $(\xx'',\yy) \in \cC_n(K)$, where $\xx'' := (x_1'',\dotsc,
  x_n'')$.
  The change of coordinates $\bm x \mapsto \bm x''$ is defined over $\ZZ$,
  does not depend on $K$, and, does not modify the $x_1$- or $\bm
  y$-coordinate, as required for (\ref{lem:coc3});
  part (\ref{lem:coc2}) follows since $\varphi_n$ is defined as a composite of maps,
  the Jacobian determinant of each of which is identically $1$.
\end{proof}

\begin{rem}
  \label{rem:jacobian}
  It follows from Lemma~\ref{lem:coc}(\ref{lem:coc2})
  that the change of variables afforded by $\varphi_n$ does not affect $p$-adic measures.
  Moreover, it is well-known that if $0\not= f \in \fO_K[X_1,\dotsc,X_n]$, then
  the zero locus of $f$ in $\fO_K^n$ has measure zero.
  We conclude that $V_n$ and $V_n'$ in Lemma~\ref{lem:coc} are without relevance
  for the computation of the integral in Proposition~\ref{prop:coneint}.
\end{rem}

\subsection{Final steps towards Theorem~\ref{thm:nilpotent}}
\label{ss:proof_nilpotent}

By combining Corollary~\ref{cor:rec} and Lemma~\ref{lem:coc}, we may reduce
the computation of the integral in Proposition~\ref{prop:coneint} for $A =
\Dual(\bm\lambda)$ to a purely combinatorial problem.

\begin{prop}
  \label{prop:make_monomial}
  Let $\bm\lambda = (\lambda_1,\dotsc,\lambda_r) \parts n$ and let $K$ be a
  $p$-adic field.
  Then
  \begin{equation}
    \label{eq:int_linear}
    \zeta_{\Dual(\bm\lambda),\fO_K}(s) = (1-q_K^{-1})^{-n} \int_{V_{\bm\lambda}(\fO_K)} \prod_{i=1}^n
    \abs{x_i}_K^{s-i} \dd\mu(\xx),
  \end{equation}
  where
  $V_{\bm\lambda}(\fO_K)$ consists of those $\xx \in \fO_K^{n(n+1)/2}$ satisfying the
  following divisibility conditions, where the $y_{i,j,\ell}$ below denote \underline{distinct}
  variables among the $x_{n+1},\dotsc,x_{n(n+1)/2}$:
  \begin{itemize}
  \item
    For $2 \le i \le r$ and $1 \le j \le \lambda_{i}$,
    \[
   {x_{\us_{i-1}(\bm\lambda) + j}}\mathrel{\Big|}{x_{\us_{i-2}(\bm\lambda)+j}, y_{i,j,1},\dotsc,y_{i,j,j-1}}.
    \]
  \item
    For $3 \le i \le r$ and $\us_{i-1}(\bm\lambda) < j \le n$,
    \[
    {x_j}\mathrel{\Big|}{y_{i,j,n+1},\dotsc,y_{i,j,n+\lambda_{i-2}}}.
    \]
  \end{itemize}
\end{prop}
\begin{rem*}
  Since the $y_{i,j,\ell}$ do not appear in the integrand in the right-hand side
  of \eqref{eq:int_linear}, it is of no consequence precisely which of the
  $x_{n+1},\dotsc,x_{n(n+1)/2}$ each $y_{i,j,\ell}$ refers to provided that
  distinct triples $(i,j,\ell)$ yield different $y_{i,j,\ell}$. 
\end{rem*}
\begin{proof}[Proof of Proposition~\ref{prop:make_monomial}]
  If $r \le 1$, the claim is trivially true so let $r \ge 2$.

  As our first step, we combine Corollary~\ref{cor:rec} and Lemma~\ref{lem:coc} in order to
  transform the membership condition (i) in \eqref{eq:Vrec} into the given
  divisibility conditions for $i = 2$ and $i =3$, respectively;
  here, $x_1,\dotsc,x_n$ correspond to the diagonal entries
  $x_{11},\dotsc,x_{nn}$ in Proposition~\ref{prop:coneint}.
  This transformation does not affect the integrand in \eqref{eq:GSS} thanks
  to condition (\ref{lem:coc3}) in Lemma~\ref{lem:coc}.

  Subsequent steps then recursively apply the same procedure
  in order to express the condition $\bm x' \in V_K(\Dual(\der{\bm\lambda}))$ in
  Corollary~\ref{cor:rec} in terms of the stated divisibility conditions,
  taking into account the evident shifts of variable indices.
  Crucially, in doing so, none of the diagonal coordinates $x_1,\dotsc,x_n$ will ever be
  modified, again thanks to condition (\ref{lem:coc3}) in Lemma~\ref{lem:coc}.
  Therefore, the divisibility conditions obtained during earlier steps will
  never be altered by subsequent ones.
  The claim thus follows by induction.
\end{proof}

\begin{proof}[Proof of Theorem~\ref{thm:nilpotent}]
  We once again omit subscripts ``$K$'' in the following.
  Moreover, we will make repeated use of the identity
  \begin{equation}
    \label{eq:x_mid_y}
    \int\limits_{\{ (x,y) \in \fO^2 : \divides x y\}} \abs{x}^r \abs{y}^s
    \dd\mu(x,y) = \int\limits_{\fO^2}\abs{x}^{r+s+1} \abs{y}^s \dd\mu(x,y)
  \end{equation}
  which follows by performing a change of variables $y = xy'$ on the
  left-hand side.
  We will furthermore use the well-known identity $\int_{\fO}\abs{x}^s\dd\mu(x)
  = (1-q^{-1})/(1-q^{-s-1})$.

  By repeatedly applying \eqref{eq:x_mid_y}, we can eliminate all the
  $y_{i,j,\ell}$ variables and rewrite~\eqref{eq:int_linear} as an integral over
  $\fO^n$.
  In order to record the effect of this procedure on the integrand, we use
  $\bm\lambda$ to index $x_1,\dotsc,x_n$ as follows.
  Let $f(i,j) := \us_{i-1}(\bm\lambda) + j$ and, for $\xx = (x_1,\dotsc,x_n)$, write
  $x_{ij} := x_{f(i,j)}$. 
  Define
  \[
  U_{\bm\lambda}(\fO) := \Bigl\{ \xx \in \fO^n : \divides{x_{i,j}}{x_{i-1,j}}
  \text{ for } 2 \le i \le r \text{ and } 1 \le j \le \lambda_i
  \Bigr\}
  \]
  Proposition~\ref{prop:permconj} and repeated applications of \eqref{eq:x_mid_y} to \eqref{eq:int_linear}
  show that
  $$
  \zeta_{\Nil(\bm\lambda^*),\fO}(s) =
  \zeta_{\Dual(\bm\lambda),\fO}(s) = (1-q^{-1})^{-n}\int\limits_{U_{\bm\lambda}(\fO)}
  F_{\bm\lambda}(\xx)\dd\mu(\xx),$$
  where
  \begin{align*}
    F_{\bm\lambda}(\xx) & = \prod_{i=1}^r \prod_{j=1}^{\lambda_i} \abs[\big]{x_{ij}}^{s -
      f(i,j)}
    \times
    \prod_{i=2}^r \prod_{j=1}^{\lambda_i} \abs[\big]{x_{ij}}^{j-1}
    \times
    \prod_{a=3}^r \prod_{i=a}^r \prod_{j=1}^{\lambda_i}
    \abs[\big]{x_{ij}}^{\lambda_{a-2}}
    \\
    & =
    \prod_{j=1}^{\lambda_1} \abs[\big]{x_{1j}}^{s-j} \times
    \prod_{i=2}^r \prod_{j=1}^{\lambda_i} \abs[\big]{x_{ij}}^{s - (\lambda_{i-1}+1)};
  \end{align*}
  the second equality follows since
  $s - f(i,j) + j - 1 + \sum_{a=3}^i \lambda_{a-2} = s - (\lambda_{i-1} + 1)$
  for $2 \le i \le r$ and $1 \le j \le \lambda_i$.
  Another sequence of applications of \eqref{eq:x_mid_y} can be used to remove
  the divisibility conditions in $U_{\bm\lambda}(\fO)$, yielding
  \pushQED{\qed}
  \begin{align*}
    (1-q^{-1})^n
  \zeta_{\Dual(\bm\lambda),\fO}(s) & =
  \int_{\fO^n} \prod_{j=1}^{\lambda_1} \abs[\big]{x_{1j}}^{s-j} \times
  \prod_{i=2}^r \prod_{j=1}^{\lambda_i} \abs[\big]{x_{ij}}^{s-j + i - 1 +
    \sum\limits_{a=1}^{i-1}(s - (\lambda_a+1))} \dd\mu(\xx)\\
  & =
  \int_{\fO^n} \prod_{i=1}^r \prod_{j=1}^{\lambda_i} \abs[\big]{x_{ij}}^{is
    - (\us_{i-1}(\bm\lambda) + j)} \dd\mu(\xx)\\
  & = (1-q^{-1})^n \dtimes \prod_{i=1}^r \prod_{j=1}^{\lambda_i} \Bigl(1-q^{-is +
    \us_{i-1}(\bm\lambda) + j - 1}\Bigr)^{-1} \\
  &= (1-q^{-1})^n \dtimes W_{\bm\lambda}(q,q^{-s}). \qedhere \popQED
  \end{align*}
\end{proof}

\section{Proofs of Theorems~\ref{Thm:global}--\ref{Thm:poles}}
\label{s:proofs}

At the heart of our proofs of Theorems~\ref{Thm:global}--\ref{Thm:poles} lies
the following local version of Theorem~\ref{Thm:global}.

\begin{thm}
  \label{thm:local}
  Let $S\subset \Places_k$ be finite and $A \in \Mat_n(\fo_S)$.
  Let $((f_1,\bm\lambda_1), \dotsc, (f_e,\bm\lambda_e))$ be an elementary divisor
  vector of $A$ over $k$.
  Write $k_i = k[X]/(f_i)$.
  Let $\fo_i$ denote the ring of integers of $k_i$.
  Then for almost all $v \in \Places_k$,
  \begin{equation}
    \label{eq:local}
    \zeta_{A,\fo_v}(s) = \prod_{i=1}^e\prod_{j=1}^{\abs{\bm\lambda_i}}
    \prod_{\substack{w \in \Places_{k_i}\\\divides{w} v}}
    \zeta_{\fo_{i,w}}\bigl( (\bm\lambda_i^*)^{-1}(j) \dtimes s - j + 1\bigr).
  \end{equation}
\end{thm}
\begin{proof}
  Combine Proposition~\ref{prop:primary},
  Theorem~\ref{thm:rednil}, and Theorem~\ref{thm:nilpotent}.
\end{proof}

The following is a consequence of Proposition~\ref{prop:coneint} and well-known
rationality results from $p$-adic integration.

\begin{prop}[{Cf.\ \cite[\S 3]{GSS88}}]
  \label{prop:rationality}
  Let $K$ be a $p$-adic field and let $A  \in \Mat_n(\fO_K)$.
  Then $\zeta_{A,\fO_K}(s) \in \QQ(q_K^{-s})$.
  Hence, $\zeta_{A,\fO_K}(s)$ admits meromorphic continuation to all of
  $\CC$.
\end{prop}

In order to deduce parts (\ref{Thm:global2})--(\ref{Thm:global3}) of
Theorem~\ref{Thm:global}, we will use the following corollary to the detailed
analysis of analytic properties of subobject zeta functions in \cite{dSG00}.
\begin{lemma}
  \label{lem:alpha}
  Let $S' \subset \Places_k$ be finite, $S\subset S'$, and let
  $A \in \Mat_n(\fo_S)$.
  Then $\alpha_{A,\fo_S} = \alpha_{A,\fo_{S'}}$
  and $\beta_{A,\fo_S} = \beta_{A,\fo_{S'}}$.
\end{lemma}
\begin{proof}
  We first argue that $\alpha_{A,\fo_v} < \alpha_{A,\fo_S}$ for each $v \in
  \Places_k\setminus S$.
  The zeta function $\zeta_{A,\fo_S}(s+n)$ is an Euler product of cone integrals
  (cf.\ Proposition~\ref{prop:coneint}) in the sense of \cite[Def.\ 4.2]{dSG00};
  cf.\ \cite[Cor.\ 5.6]{dSG00}.
  Using the notation from \cite{dSG00},
  by \cite[Cor.\ 3.4]{dSG00} (which is correct despite a minor, fixable mistake
  in \cite[Prop.\ 3.3]{dSG00}, see \cite[Rem.\ 4.6]{AKOV13}),
  it follows that each $\alpha_{A,\fo_v}$ for $v \in
  \Places_k\setminus S$ is a number of the form $n-B_j/A_j$ for $j = 1,\dotsc,q$.
  Hence, by combining \cite[Cor.\ 4.14, Lem.\ 4.15]{dSG00},
  for each $v \in \Places_k\setminus S$,
  $$\alpha_{A,\fo_v} < n + \max_{k=1,\dotsc,q} \frac{1-B_k}{A_k} =
  \alpha_{A,\fo_S}.$$

  Clearly, $0 < \alpha_{A,\fo_{S'}} \le \alpha_{A,\fo_S}$.
  Define $F(s) = \prod_{v\in S'\setminus S}\zeta_{A,\fo_v}(s)$ 
  so that
  $\zeta_{A,\fo_S}(s) = F(s)\zeta_{A,\fo_{S'}}(s)$
  for all $s \in \CC$ with $\mathrm{Re}(s) > \alpha_{A,\fo_S} - \delta$ and some
  constant $\delta > 0$ (see Theorem~\ref{thm:dSG}).
  By the above, every real pole of $F(s)$ is less than $\alpha_{A,\fo_S}$.
  Since $F(s)$ is a non-zero Dirichlet series with non-negative coefficients,
  we conclude that $F(\alpha_{A,\fo_S}) > 0$.
  In particular, since $\zeta_{A,\fo_S}(s)$ has a pole at $\alpha_{A,\fo_S}$,
  the same is true of $\zeta_{A,\fo_{S'}}(s)$  whence $\alpha_{A,\fo_{S'}} \ge
  \alpha_{A,\fo_S}$.
  Moreover, $F(\alpha_{A,\fo_S}) > 0$ clearly also implies that $\beta_{A,\fo_S} = \beta_{A,\fo_{S'}}$.
\end{proof}
\begin{rem}
  \label{rem:alphainv}
  \quad
\begin{enumerate}
  \item
  \label{rem:alphainv1}
    The corresponding statement for subalgebra and submodule zeta functions
    (proved in the same way) is certainly well-known to experts in the area.
    Unfortunately, it does not seem to have been spelled out in the literature.
    For a similar statement in the context of representation zeta functions,
    see~\cite[Thm~1.4]{AKOV16}.
  \item
    While in \cite{dSG00} only the case $k = \QQ$, $S = \emptyset$ is discussed,
    their arguments carry over to the present setting in the expected way
    (cf.\ \cite{AKOV13} and \cite[\S 4]{DV15}).
  \end{enumerate}
\end{rem}

\begin{proof}[Proof of Theorem~\ref{Thm:global}]
  Part (\ref{Thm:global1}) follows from Theorem~\ref{thm:local} and Proposition~\ref{prop:rationality}.
  Let $\bm\mu \parts n$. 
  We now determine the largest real pole,
  $\alpha$ say, and its multiplicity, $\beta$ say, of
  \[
  \Zeta(s) := \prod_{j=1}^n \zeta_{\fo_S}(\bm\mu^{-1}(j) \dtimes s - j + 1).
  \]
  Write $r = \len{\bm\mu}$.
  Since $\zeta_{\fo_S}(s)$ has a unique pole at $1$ (with multiplicity $1$) and
  $\zeta_{\fo_S}(s_0) \not= 0$ for real $s_0 > 1$,
  \begin{align*}
  \alpha &
  = \max_{1\le j \le n} \frac j {\bm\mu^{-1}(j)}
  = \max_{1\le i \le r} \max_{1\le j \le \lambda_i} \frac {\us_{i-1}(\bm\mu) +j} i
  = \max_{1\le i \le r} \frac {\us_{i}(\bm\mu)} i = \mu_1 = \len{\bm\mu^*},
  \end{align*}
  where the penultimate equality follows since
  $i \mu_{i+1} \le \us_i(\bm\mu)$
  and thus $\frac {\us_i(\bm\mu)}  i \ge
  \frac{\us_{i+1}(\bm\mu)}{i+1}$ for $1\le i \le r - 1$.
  Next, $\beta$ is precisely the 
  number of $i \in \{1,\dotsc,r\}$ with $\mu_1 = \frac{\us_i(\bm\mu)} i$ or,
  equivalently, the largest $\ell \ge 1$ with $\mu_1 = \dotso = \mu_\ell$.
  In other words, $\beta = \mu^*_{-1}$.

  Parts (\ref{Thm:global2})--(\ref{Thm:global3}) of Theorem~\ref{Thm:global}
  now follow from Lemma~\ref{lem:alpha} and the observation that
  $\Zeta(s) > 0$ for $s > \alpha$.
\end{proof}

\begin{ex}
  \label{ex:exceptional}
  The presence of the exceptional factors $W_u(q_{w_u}^{-s})$ in Theorem~\ref{Thm:global}
  is in general unavoidable.
  For a simple example, let $a \in \fo$ be non-zero and define $A
  = \bigl[\begin{smallmatrix} 0 & a \\ 0 & 0 \end{smallmatrix}\bigr]$. 
  Using Proposition~\ref{prop:coneint}, a simple computation reveals that for $v \in \Places_k$,
  \begin{equation}
    \label{eq:exceptional}
    \zeta_{A,\fo_v}(s) = \frac{1 - q_v^{1-2s} + q_v^{(1-s)(v(a)+1)} \dtimes (q_v^{-s}-1)}{1-q_v^{1-s}} \dtimes \zeta_{\fo_v}(s) \zeta_{\fo_v}(2s - 1);
  \end{equation}
  note that $\zeta_{A,\fo_v}(s) = \zeta_{\fo_v}(s) \zeta_{\fo_v}(2s - 1)$
  whenever $v(a) = 0$.
  We further note that the exceptional factor in \eqref{eq:exceptional}
  in fact belongs to $\ZZ[q_v^{-s}]$ and is thus regular at $s = 1$;
  this is consistent with the general fact that for subobject zeta functions,
  each local abscissa of convergence is strictly less than the associated global
  one (see the proof of Lemma~\ref{lem:alpha}). 
  Finally note the failure of \eqref{eq:feqn} for the finitely many $v \in \Places_k$ with $v(a) > 0$.
\end{ex}

\begin{rem}
  In view of a conjecture of Solomon proved by Bushnell and Reiner~\cite{BR80},
  it is natural to ask if the $W_u \in \QQ(X)$ in Theorem~\ref{Thm:global} are
  in fact always elements of $\ZZ[X]$.
\end{rem}

\begin{proof}[Proof of Theorem~\ref{Thm:FEqn}]
  The claim follows by combining Theorem~\ref{thm:local} and the following simple
  observation.
  Let $k'/k$ be an extension of number fields, let $\fo'$ be the ring of
  integers of~$k'$, and let $v \in \Places_k$ be unramified in $k'$.
  If $w \in \Places_{k'}$ divides $v$,
  define $\mathfrak f(w/v)$ by $q_w = q_v^{\mathfrak f(w/v)}$.
  Define
  \[
  \Zeta_v(s) = \prod_{\substack{w \in \Places_{k'}\\\divides w v}}
  \zeta_{\fo'_w}(s)
  = \prod_{\substack{w \in \Places_{k'}\\\divides w v}} \bigl (1-q_v^{-\mathfrak
    f(w/v)s} \bigr)^{-1}.
  \]
  Then, recalling the definition of $\noplaces_v(k')$ from p.\
  \pageref{lab:noplaces}
  and using $\sum\limits_{\divides w v} \mathfrak f(w/v) = \idx{k':k}$,
  \[
  \Zeta_v(s) \Bigm\vert_{q_v\to q_v^{-1}} = (-1)^{\noplaces_v(k')}
  q_v^{-\idx{k':k} s} \dtimes \Zeta_v(s). \popQED
  \]
\end{proof}

\begin{lemma}
  Let $S\subset \Places_k$ be finite.
  Let $\Zeta(s)$ and $\Zeta'(s)$ be two Dirichlet series with finite abscissae of
  convergence.
  Suppose that $\Zeta(s) = \prod_{v \in \Places_k\setminus S}\Zeta_v(s)$ and
  $\Zeta'(s) = \prod_{v \in \Places_k\setminus S}\Zeta'_v(s)$, where
  each $\Zeta_v(s)$ and $\Zeta'_v(s)$ is a series in $q_v^{-s}$ with non-negative
  real coefficients. 
  Suppose that $\Zeta(s) = \Zeta'(s)$ and that $W(X,Y),W'(X,Y) \in \QQ(X,Y)$ satisfy
  $\Zeta_v(s) = W(q_v,q_v^{-s})$ and $\Zeta'_v(s) = W'(q_v,q_v^{-s})$ for almost
  all $v\in \Places_k \setminus S$.
  Then $W(X,Y) = W'(X,Y)$.
\end{lemma}
\begin{proof}
  Let $S_0$ be the set of rational primes which are divisible by at least one element of~$S$.
  For a rational prime~$p\not\in S_0$,
  define $\Zeta_p(s) = \prod_{{v \in
      \Places_k, \divides v p}} \Zeta_v(s)$ and define $\Zeta'_p(s)$ in the same
  way.
  Assuming that $\Zeta(s) = \Zeta'(s)$, 
  it is well-known that the coefficients of the Dirichlet series $\Zeta(s)$ and
  $\Zeta'(s)$ coincide.
  We conclude that $\Zeta_p(s) = \Zeta'_{p}(s)$ for $p \not\in S_0$.
  By Chebotarev's density theorem, there exists an infinite set of rational
  primes $P$ such that each $p\in P$ splits completely 
  in $k$.
  Writing $d = \idx{k:\QQ}$, for almost all $p \in P$,
  we thus have $W(p,p^{-s})^d = \Zeta_p(s) = \Zeta'_p(s) = 
  W'(p,p^{-s})^d$ which easily implies $W(X,Y)^d = W'(X,Y)^d$.
  Thus, $W(X,Y)/W'(X,Y)$ is a $d$th root of unity in $\RR(X,Y)$ and hence in
  $\RR$, for the latter is algebraically closed in the former (see \cite[Prop.\
  11.3.1]{Coh03}).
  The non-negativity assumptions on the coefficients of $\Zeta_v(s)$ and
  $\Zeta'_v(s)$ as series in $q_v^{-s}$ now imply $W(X,Y) = W'(X,Y)$.
\end{proof}

\begin{proof}[Proof of Theorem~\ref{Thm:sim}]
The implications
``(\ref{Thm:sim1})$\Rightarrow$(\ref{Thm:sim2})$\Rightarrow$(\ref{Thm:sim3})''
in Theorem~\ref{Thm:sim} are obvious.
Suppose that (\ref{Thm:sim3}) holds.
Let $\bm\lambda$ and $\bm\mu$ be the types of the matrices $A$ and $B$,
respectively.
By Theorem~\ref{thm:local} and the preceding lemma, 
$W_{\bm\lambda}(X,Y) = W_{\bm\mu}(X,Y)$.
It is easy to see that the binomials $1-X^aY^b$ for $a \ge 0$ and $b \ge 1$
freely generate a free abelian subgroup of $\QQ(X,Y)^\times$.
Hence, $\bm\lambda = \bm\mu$ and $A$ and $B$ are similar.
\end{proof}

\begin{rem}
  \label{rem:nilpotency_required}
  If $A$ is nilpotent and $\alpha \in k^\times$, then $A$ and $A + \alpha 1_n$
  give rise to  the same local and global zeta functions without $A$ and $A +
  \alpha 1_n$ being similar.
  In general, equality of local and global zeta functions associated with
  non-nilpotent   matrices $A$ and $B$ does not suffice to even conclude that
  the algebras $k[A]$ and $k[B]$ are similar.
  We give two examples to illustrate this behaviour, the first being arithmetic
  and the second of combinatorial origin.
  \begin{enumerate}
    \item By \cite{Per77}, there are monic irreducible polynomials $f,g\in
      \ZZ[X]$ of the same degree such that the number fields $\QQ[X]/(f)$ and
      $\QQ[X]/(g)$ are non-isomorphic but have the same Dedekind zeta
      functions; moreover, as explained in \cite[\S 1]{Per77}, every rational prime
      has the same ``splitting type''  in each of these two number fields.
      Consequently, $\zeta_{\CM(f),\ZZ_p}(s) = \zeta_{\CM(g),\ZZ_p}(s)$ for almost
      all primes~$p$ 
    \item
      Recall the definition of $W_{\bm\lambda}$ from \S\ref{s:nilpotent}.
      A simple calculation shows that $$W_{(2,2,1)} \dtimes W_{(3,1)} =
      W_{(2,2)} \dtimes W_{(3,1,1)}.$$
      Let $a,b \in k^\times$  be distinct and choose $A,B \in \Mat_9(k)$
      to have elementary divisor vectors $((X-a,(3,2), (X-b,(2,1,1)))$ and
      $((X-a,(2,2)),(X-b,(3,1,1)))$, respectively.
      Then $k[A]$ and $k[B]$ are not similar but $\zeta_{A,\fo_v}(s) =
      \zeta_{B,\fo_v}(s)$ for almost all $v \in \Places_k$.
    \end{enumerate}
\end{rem}

\begin{rem}
We further note that even for nilpotent $A$, the family of associated functional
equations \eqref{eq:feqn} in Theorem~\ref{Thm:FEqn} does not determine $A$ up to
similarity; an example is given by two nilpotent $7\times 7$-matrices with types
$(3,1,1,1,1)$ and $(2,2,2,1)$, respectively.
\end{rem}

\begin{proof}[Proof of Theorem~\ref{Thm:poles}]
  By Theorem~\ref{thm:local},
  $\zeta_{A,\fo_v}(s)$ has a pole at zero for almost all $v \in \Places_k$.
  Moreover, again for almost all $v \in \Places_k$, this pole is simple if and
  only if $e = 1$ and almost all places of $k$ remain inert in $k[X]/(f_1)$;
  the latter condition is equivalent to $f_1$ being linear.
\end{proof}

\section{Applications}
\label{s:app}

\subsection{Submodules for unipotent groups}

Let $S \subset \Places_k$ be finite,
let $\mathsf M$ be a finitely generated $\fo_S$-module, and let $\Omega
\subset \End_{\fo_S}(\mathsf M)$.
We let $\alpha_{\Omega \acts \mathsf M}$ denote the abscissa of convergence of
$\zeta_{\Omega\acts \mathsf M}(s)$.
As a special case (cf.\ \cite[Rem.~2.2(ii)]{topzeta})), given a possibly
non-associative $\fo_S$-algebra $\mathsf A$ whose underlying $\fo_S$-module is
finitely generated, we let $\alpha_{\mathsf A}$ denote the abscissa of
convergence of its ideal zeta function $\zeta_{\mathsf A}(s)$, as defined in the introduction.
We now illustrate how Theorem~\ref{Thm:global} can sometimes be used to
determine $\alpha_{\Omega \acts \mathsf M}$ or $\alpha_{\mathsf A}$ without computing
the corresponding zeta function.
The key observation is that if $\omega \in \Omega$, then $\alpha_{\Omega \acts
  \mathsf M} \le \alpha_{\omega,\fo}$;
by Theorem~\ref{Thm:global}(\ref{Thm:global2}),
the latter number can be easily read off from an elementary divisor vector of
$\omega \otimes_{\fo_S} k$.

We let $\Uni_n$ denote the group scheme of upper unitriangular $n\times n$
matrices.
For $\bm\lambda = (\lambda_1,\dotsc,\lambda_r) \parts n$, we regard $\Uni_{\bm\lambda} :=
\Uni_{\lambda_1}\times \dotsb \times \Uni_{\lambda_r}$ as a subgroup scheme of
$\Uni_n$ via the natural diagonal embedding.
The case $\len{\bm\lambda} = 1$ of the following provides an affirmative answer to \cite[Question~9.7]{padzeta}.

\begin{prop}
  Let $\bm\lambda \parts n$.
  Then $\alpha_{\Uni_{\bm\lambda}(\fo) \acts \fo^n} = \len{\bm\lambda}$.
\end{prop}
\begin{proof}
  Using the characterisation of $\Uni_m(k)$ as the centraliser of a maximal
  flag of subspaces of $k^m$, we see that $\fo^n$ contains an
  $\Uni_{\bm\lambda}(\fo)$-invariant submodule $N$ such that $\Uni_{\bm\lambda}(\fo)$ acts
  trivially on $\fo^n/N$ and $\fo^n/N \approx_{\fo}  \fo^{\len{\bm\lambda}}$.
  We conclude that $\alpha_{\Uni_{\bm\lambda}(\fo) \acts \fo^n} \ge \len{\bm\lambda}$.
  For an upper bound, note that $(1 + \Nil(\bm\lambda)) \in \Uni_{\bm\lambda}(\fo)$ 
  whence $\alpha_{\Uni_{\bm\lambda}(\fo) \acts \fo^n} \le \alpha_{\Nil(\bm\lambda),\fo}
  = \len{\bm\lambda}$.
\end{proof}

For $\abs{\bm\lambda} \le 5$ and almost all $v \in \Places_k$,
explicit formulae for $\zeta_{\Uni_{\bm\lambda}(\fo_v)\acts \fo_v^n}(s)$ have been
obtained by the author (see \cite[\S 9.4]{padzeta} and the database
included with \cite{Zeta});
the only unknown case for $\len{\bm\lambda} = 6$, namely 
$\bm\lambda = (6)$, seems out of reach at present.
In addition to their global abscissae of convergence, the
$\zeta_{\Uni_{\bm\lambda}(\fo_v)\acts \fo_v^n}(s)$ are known to
generically satisfy local functional equations under inversion of $q_v$ by
\cite[\S 5.2]{Vol16}.

\subsection{Lie algebras of maximal class}

Let $\bm\fg$ be a finite-dimensional Lie $k$-algebra.
For finite $S\subset \Places_k$, by an \emph{$\fo_S$-form} of $\bm\fg$,
we mean a Lie $\fo_S$-algebra $\fg$ whose underlying module is free and such
that $\fg \otimes_{\fo_S} k \approx_k \bm\fg$.

Let $\bm\fg = \bm\fg^1 \supset \bm\fg^2\supset \dotsb$ be the lower central
series of $\bm\fg$.
Recall that $\bm\fg$ has \emph{maximal class} if $\bm\fg$ is nilpotent of class
$\dim_k(\bm\fg)-1$.
Equivalently, $\bm\fg$ has maximal class if and only if $\dim_k(\bm\fg^1/\bm\fg^{2}) = 2$ and $\dim_k(\bm\fg^i/\bm\fg^{i+1}) =
1$ for $1 \le i \le \dim_k(\bm\fg) - 1$.

\begin{prop}
  \label{prop:maxclass}
  Let $\fg$ be an $\fo_S$-form of a non-abelian finite-dimensional Lie
  $k$-algebra of maximal class. 
  Then $\alpha_{\fg} = 2$.
\end{prop}

A proof of Proposition~\ref{prop:maxclass} using Theorem~\ref{Thm:global} will be given below.

We note that Proposition~\ref{prop:maxclass} is consistent with explicit calculations
carried out for specific $\ZZ$-forms of the 
Lie algebras $M_3$,$M_4$,$M_5$, and $\mathrm{Fil}_4$ of maximal class and
dimension at most $5$ over the rationals; see \cite[Ch.\ 2]{dSW08}.

\begin{lemma}
  Let $S\subset \Places_k$ be finite.
  Let $\fg$ be an $\fo_S$-form of a nilpotent Lie $k$-algebra of finite
  dimension $n$.
  Let $\mathsf A$ be the enveloping unital associative algebra of
  $\mathrm{ad}(\fg)$ within $\End_{\fo_S}(\fg)$.
  \begin{enumerate}
    \item For each $\varphi \in \mathsf A$, there exists $c \in \fo_S$ with $(\varphi - c
      1_{\fg})^n = 0$; thus, $\varphi \otimes_{\fo_S} k$ is primary.
    \item
      Let $\varphi \in \mathsf A$ have type $\bm\lambda$ over $k$.
      Then $\alpha_{\fg} \le \len{\bm\lambda}$.
    \end{enumerate}
\end{lemma}
\begin{proof}
  The first part follows from Engel's theorem and the second part is then an
  immediate consequence of Theorem~\ref{Thm:global}(\ref{Thm:global2}).
\end{proof}

\begin{lemma}
  \label{lem:goodbasis}
  Let $\bm\fg$ be an $(n+2)$-dimensional non-abelian Lie $k$-algebra of maximal class.
  Then there exists a $k$-basis $(x_1,x_2,y_1,\dotsc,y_n)$ of $\bm\fg$ such that
  $[x_1,x_2] = y_1$, 
  $[x_1,y_i] = y_{i+1}$ for $1 \le i \le n-1$, and $[x_1,y_n] = 0$.
\end{lemma}
\begin{proof}
  Consider the graded Lie algebra $\bigoplus_{i\ge 1} \bm\fg^i/\bm\fg^{i+1}$
  associated with $\bm\fg$.
  We claim that there exists an element $a \in \bm\fg/\bm\fg^2$ such 
  that $[a,\dtimes]$ maps $\bm\fg^i/\bm\fg^{i+1}$ onto
  $\bm\fg^{i+1}/\bm\fg^{i+2}$ for each $i \ge 1$.
  To see that, first note that $[\bm\fg/\bm\fg^2,\bm\fg^i/\bm\fg^{i+1}] =
  \bm\fg^{i+1}/\bm\fg^{i+2}$ for each $i \ge 1$.
  Let $(u,v)$ be a $k$-basis of $\bm\fg/\bm\fg^2$.
  Then $[u,v]$ spans $\bm\fg^2/\bm\fg^3$.
  Moreover, if $w_i$ spans $\bm\fg^i/\bm\fg^{i+1}$,
  then the image of at least one of $[u,w_i]$ and $[v,w_i]$ spans
  $\bm\fg^{i+1}/\bm\fg^{i+2}$.
  Consequently, we may take $a = u + cv$ for almost all $c \in k$.

  Given $a$ as above, choose $b \in \bm\fg/\bm\fg^2$ such that $(a,b)$ is a basis of
  $\bm\fg/\bm\fg^2$.
  Let $x_1, x_2 \in \bm\fg$ be preimages of $a$ and $b$, respectively.
  Then, if we define $y_1 = [x_1,x_2]$ and $y_{i+1} = [x_1,y_i]$, we obtain a
  basis $(x_1,x_2,y_1,\dotsc,y_n)$ of the desired form.
\end{proof}

\begin{proof}[Proof of Proposition~\ref{prop:maxclass}]
  By Lemma~\ref{lem:alpha} and Remark~\ref{rem:alphainv}(\ref{rem:alphainv1}),
  we are free to enlarge $S$ as needed.
  In particular, we may assume that $\fg/\fg^2 \approx_{\fo_S} \fo_S^2$ whence
  $\alpha_{\fg} \ge 2$ follows.
  Moreover, we may assume
  that $\fg$ possesses an $\fo_S$-basis $(x_1,x_2,y_1,\dotsc,y_n)$ as in by
  Lemma~\ref{lem:goodbasis}.
  The matrix of $[x_1,\dtimes]$ with respect to the basis
  $(x_2,y_1,\dotsc,y_n,x_1)$ is precisely $\Nil((n+1,1))$ whence
  $\alpha_{\fg} \le 2$ follows from Theorem~\ref{Thm:global}.
\end{proof}

{
  \bibliographystyle{abbrv}
  \tiny
  \bibliography{cyclic}
}

\end{document}